\documentclass[twoside,12pt]{amsart}

\usepackage{amsmath} 
\usepackage{calrsfs} 
\usepackage{bbm} 
\usepackage{fancyhdr}  
\usepackage{amssymb} 
\usepackage[all,cmtip]{xy} 
\usepackage[margin=1in]{geometry} 

\usepackage{aliascnt} 

\theoremstyle{plain}
\newtheorem{theorem}{Theorem}[section]
\newaliascnt{corollary}{theorem}
\newtheorem{corollary}[corollary]{Corollary}
\aliascntresetthe{corollary}

\newaliascnt{lemma}{theorem}
\newtheorem{lemma}[lemma]{Lemma}
\aliascntresetthe{lemma}
\newaliascnt{proposition}{theorem}
\newtheorem{proposition}[proposition]{Proposition}
\aliascntresetthe{proposition}

\newaliascnt{hypotheses}{theorem}
\newtheorem{hypotheses}[hypotheses]{Hypotheses}
\aliascntresetthe{hypotheses}

\theoremstyle{definition}
\newaliascnt{definition}{theorem}
\newtheorem{definition}[definition]{Definition}
\aliascntresetthe{definition}

\newaliascnt{example}{theorem}

\aliascntresetthe{example}

\newaliascnt{remark}{theorem}

\aliascntresetthe{remark}

\newaliascnt{remarks}{theorem}
\newtheorem{remarks}[remarks]{Remarks}
\aliascntresetthe{remarks}

\numberwithin{equation}{section}

\newcommand\shtitle{Values of singular currents}
\newcommand\shortauthor{Luis E. Garcia}

\everymath{\displaystyle}

\usepackage[numbers]{natbib}
\usepackage[pdftex,breaklinks,colorlinks,
citecolor=blue,
linkcolor=blue,
urlcolor=blue]{hyperref}

\begin{document}
\title{On the values of some singular currents on Shimura varieties of orthogonal type}
\author{Luis E. Garcia}
\date{}

\begin{abstract}
In \cite{GarciaRegularizedLiftsI}, we introduced a regularized theta lift for reductive dual pairs of the form $(Sp_4,O(V))$ with $V$ a quadratic vector space over a totally real number field $F$. The lift takes values in the space of $(1,1)$-currents on the Shimura variety attached to $GSpin(V)$, and we proved that its values are cohomologous to currents given by integration on special divisors against automorphic Green functions. In this paper, we will evaluate the regularized theta lift on differential forms obtained as usual (non-regularized) theta lifts. Using the Siegel-Weil formula and ideas of Piatetskii-Shapiro and Rallis, we show that the result involves near central special values of standard $L$-functions for $Sp_{4,F}$. An example concerning products of Shimura curves will be given at the end of the paper. 
\end{abstract}
\maketitle
\setcounter{tocdepth}{2}
\tableofcontents


%
%
%
%

\section{Introduction} \label{section:Introduction} \subsection{Background and main results} In the paper \cite{GarciaRegularizedLiftsI}, we considered some $(1,1)$-currents on GSpin Shimura varieties; in this paper we will evaluate those currents on certain differential forms. We fix a quadratic vector space $V$ over a totally real field $F$ with signature of the form $(n,2)$ at one real place of $F$ and positive definite at all other real places, with $n$ positive and even. If we denote by $H=Res_{F/\mathbb{Q}}GSpin(V)$ the restriction of scalars of the algebraic group $GSpin(V)$, then the complex points of the Shimura variety $X_K$ attached to a neat open compact subgroup $K \subset H(\mathbb{A}_f)$ are given by
\begin{equation}
X_K(\mathbb{C}) = H(\mathbb{Q}) \backslash \mathbb{D} \times H(\mathbb{A}_f) /K,
\end{equation}
where $\mathbb{D}$ denotes the hermitian symmetric domain attached to the Lie group $SO(n,2)$. The currents in question are denoted by 
\begin{equation}
[\Phi(T,\varphi)] \in \mathcal{D}^{1,1}(X):=\varprojlim_K \mathcal{D}^{1,1}(X_K)
\end{equation} 
and are parametrized a pair $(T,\varphi)$ consisting of a totally positive definite symmetric matrix $T \in Sym_2(F)$ and a Schwartz function $\varphi \in \mathcal{S}(V(\mathbb{A}_f)^2)$. Their interest lies in the fact that their $\mathbb{Q}$-linear span includes many currents of the form
\begin{equation}
2\pi i \log|f| \cdot \delta_{Y},
\end{equation}
where $Y$ is a special Shimura subvariety of $X$ and $f$ is one of the meromorphic functions on $Y$ constructed in \cite{Bruinier}. Using a regularized theta lift, we also constructed a different family of currents $[\tilde{\Phi}(T,\varphi,s)]$ depending on a complex parameter $s$ and related the two types of currents (see \cite[Theorem 1.1]{GarciaRegularizedLiftsI}).

The present paper is concerned with evaluating the currents $[\tilde{\Phi}(T,\varphi,s)]$ on differential forms. Having constructed these currents as regularized theta lifts for the reductive dual pair $(Sp_4,O(V))$, it seems reasonable to try to evaluate them on differential forms that also arise as usual (i.e. absolutely convergent) theta lifts for this dual pair. Let us briefly describe these forms: in \autoref{subsection:theta_lifts_diff_forms} we review the construction of the theta function
\begin{equation}
\theta(g;\varphi \otimes \varphi_\infty) \in \mathcal{A}^{n-1,n-1}(X):=\varinjlim_K \mathcal{A}^{n-1,n-1}(X_K)
\end{equation}
attached to a Schwartz form
\begin{equation}
\varphi \otimes \varphi_\infty \in [\mathcal{S}(V(\mathbb{A})^2) \otimes \mathcal{A}^{n-1,n-1}(\mathbb{D})]^{H(\mathbb{R})}.
\end{equation} 
Here $g \in Sp_4(\mathbb{A}_F)$ and the function $\theta( \cdot \ ;\varphi \otimes \varphi_\infty)_K$ is of moderate growth and left invariant under $Sp_4(F)$. Given a cusp form $f \in \mathcal{A}_0(Sp_{4,F})$, the integral
\begin{equation}
\theta(f,\varphi \otimes \varphi_\infty)=\int_{Sp_4(F) \backslash Sp_4(\mathbb{A}_F)}\overline{f(g)}\theta(g;\varphi \otimes \varphi_\infty) dg
\end{equation}
converges due to the rapid decrease of $f$ and defines a differential form in $\mathcal{A}^{n-1,n-1}(X)$ known as a theta lift of $f$. The properties of these theta lifts have been extensively studied, and in particular the question of when theta lifts are not identically zero is known to be related to special values of automorphic $L$-functions. More precisely, assume that $f$ generates an irreducible cuspidal automorphic representation $\pi \subset \mathcal{A}_0(Sp_{4,F})$. One can interpret the theta lift $\theta(f,\varphi \otimes \varphi_\infty)$ as an automorphic form $\theta_\varphi(f)$ on $O(V)(\mathbb{A})$, and a fundamental theorem due to Rallis \cite{Rallis1982} states that the Petersson inner product $(\theta_\varphi(f),\theta_\varphi(f))_{Pet}$ is, up to an explicit non-zero constant, equal to a finite product of local zeta integrals times a special value of the incomplete standard $L$-function $L^S(\pi \times \chi_V,std,s)$. In particular, this formula, widely known as the Rallis inner product formula, often implies that the theta lift is not identically zero.

Our main result in this paper will be a computation of the values of $[\tilde{\Phi}(T,\tilde{\varphi},s)]$ on the forms $\theta(f,\varphi \otimes \varphi_\infty)$ that resembles the Rallis inner product formula just described. This is the content of \autoref{thm:Rallis_pairing} that we restate below. To obtain a formula involving a special value of a complete $L$-function, it is necessary to introduce some extra data: we choose an open compact subgroup $K^\infty \subset Sp_4(\mathbb{A}_F^\infty)$ and an irreducible representation $\sigma^\infty$ of $K^\infty$. For vectors $w \in \sigma^\infty$ and $w^\vee \in (\sigma^\infty)^\vee$, we denote by $\varphi_{w^\vee,w}$ the Schwartz form in $\mathcal{S}(V(\mathbb{A}_F^\infty)^2)$ obtained by integrating $\varphi$ against the matrix coefficient of $(\sigma^\infty)^\vee$ determined by $w$ and $w^\vee$.

\begin{theorem} \label{thm:Rallis_pairing_simplified} Let $\pi \subset \mathcal{A}_0(Sp_{4,F})$ be a cuspidal automorphic representation of $Sp_4(\mathbb{A}_F)$ and $f \in \pi$ be a cusp form. Let $\sigma^\infty$ be an irreducible representation of an open compact subgroup $K^\infty$ of $Sp_4(\mathbb{A}_F^\infty)$ and let $\sigma$ be the representation of $K=K^\infty \times U(2)^d$ given by \eqref{eq:def_sigma}. Assume that the hypotheses in \ref{hyp:main_thm_pairing} hold for $f$, $\sigma^\infty$ and $\pi$ and for the Schwartz forms $\tilde{\varphi}$, $\varphi$. Choose a $K$-invariant embedding $\iota_\pi=\otimes_v \iota_{\pi,v}: \sigma \hookrightarrow \pi$. Let $w=\otimes_{v \nmid \infty} w_v \in \sigma^\infty$, $w^\vee=\otimes_{v \nmid \infty} w_v^\vee \in (\sigma^\infty)^\vee$ and $f^\vee=\otimes_v f_v^\vee \in \pi^\vee$ such that $(\iota_{\pi,v}(w_v),f_v^\vee)=1$ for all $v$. Then there is a positive constant $C$, depending only on the quadratic space $V$, such that
\begin{equation*}
\begin{split}
([\tilde{\Phi}(T,\tilde{\varphi}_{w^\vee,w},s')],\overline{\theta(f,\varphi \otimes \varphi_\infty)}) & =C \cdot (\widetilde{\mathcal{M}}_T(s'),f_{\iota_\pi(w \otimes w_\infty)})^{reg} \\
& \quad \cdot \left. \frac{\Lambda(\pi^\vee,\chi_V,s+1/2)}{d(\chi_V,s)} \right|_{s=s_0} \cdot \prod_v Z^0_v(f^\vee_v,f_v,\delta_*\Phi_{w_v^\vee,w_v},\chi_{V,v},s_0),
\end{split}
\end{equation*}
where $s_0=(n-3)/2$ and for all but finitely many $v$ we have $Z^0_v(f^\vee_v,f_v,\delta_*\Phi_{w_v^\vee,w_v},\chi_{V,v},s_0)=1$.
\end{theorem}

In this statement, the pairing in the left hand side is defined by \eqref{eq:def_pairing_limit}, $\Lambda(\pi^\vee,\chi_V,s)$ denotes the complete $L$-function defined by the doubling method, $d(\chi_V,s)$ is an explicit product of completed degree $1$ $L$-functions given by \eqref{eq:global_d_computation}, the factors $Z^0_v(\cdot)$ are normalized local zeta integrals on $Sp_4(F_v)$ and $(\widetilde{\mathcal{M}}_T(s'),f)^{reg}$ is the regularized integral defined by \eqref{eq:def_theta_lift} with $f_{\iota_\pi(w \otimes w_\infty)}$ the cusp form in $\pi$ corresponding to $\iota_\pi(w \otimes w_\infty)$. The right hand side does not depend on the choice of embedding $\iota_\pi$.

Theorem \ref*{thm:Rallis_pairing_simplified} also yields information about the values of the currents $[\Phi(T,\varphi)]$ on closed forms. Namely, we show in \autoref{subsection:reps_and_cohomology} that the hypotheses in \ref{hyp:main_thm_pairing} imply that the form $\theta(f,\varphi \otimes \varphi_\infty)$ is closed. Assuming that $\tilde{\varphi}_{w^\vee,w}^\iota=\tilde{\varphi}_{w^\vee,w}$, the results in \cite{GarciaRegularizedLiftsI} imply that
\begin{equation*}
\begin{split}
([\Phi(T,\tilde{\varphi}_{w^\vee,w})]-[\Phi(T^\iota,\tilde{\varphi}_{w^\vee,w}^\iota)],\overline{\theta(f,\varphi\otimes \varphi_\infty)})&=C \cdot \left. \frac{\Lambda(\pi^\vee,\chi_V,s+1/2)}{d(\chi_V,s)}  \right|_{s=s_0} \\
& \quad \cdot (I(T,\tilde{\varphi}_{w^\vee,w};f,\varphi)-I(T^\iota,\tilde{\varphi}_{w^\vee,w};f,\varphi)),
\end{split}
\end{equation*}
where $I(T,\tilde{\varphi}_{w^\vee,w};f,\varphi)$ is defined in \eqref{eq:def_Beil_period}; see \autoref{cor:thm_Rallis_pairing}.

Let us consider an example that provides motivation for the appearance of the group $Sp_4$ in the results discussed above. Namely, consider the case 
\begin{equation}
X_K=X_0^B \times X_0^B,
\end{equation}
where $X_0^B$ is the full level Shimura curve attached to a non-split, indefinite quaternion algebra $B$ over $\mathbb{Q}$. Then the cohomology group $H^{1,1}(X_K)$ is naturally a module over the full level Hecke algebra of $PB(\mathbb{A}_f)^\times \times PB(\mathbb{A}_f)^\times$. The Matsushima formula shows that this module splits as a direct sum of irreducible representations, indexed by pairs of automorphic representations $(\pi_1,\pi_2)$ of $PB(\mathbb{A})^\times$ (with appropriate conductor and archimedean components); denote by $H^{1,1}(X_K)[\pi_1,\pi_2]$ the summand of $H^{1,1}(X_K)$ corresponding to such a pair.

The behaviour of $H^{1,1}(X_K)[\pi_1,\pi_2]$ with respect to algebraic cycles is then conjecturally controlled by the Rankin-Selberg $L$-function $L(\pi_1 \times \pi_2,s)$. It is known that there is a dichotomy in the behaviour of this $L$-function: namely, it has a (simple) pole at $s=1$ if and only if $\pi_1 \cong \pi_2^\vee$. If such a pole exists, then the Tate conjecture predicts the existence of an element of $CH^{1}(X_K)$ whose cohomology class projects non-trivially to $H^{1,1}(X_K)[\pi_1,\pi_2]$; in fact, the diagonal $X_0^B \rightarrow (X_0^B)^2$ has this property. In the absence of such a pole, i.e. when $\pi_1 \ncong \pi_2^\vee$, a conjecture of Beilinson predicts the existence of a higher Chow cycle $Z \in CH^2(X,1)_\mathbb{Z}$ (see \cite[\textsection 3.9]{GarciaRegularizedLiftsI} for definitions) whose regulator has a non-trivial $H^{1,1}(X_K)[\pi_1,\pi_2]$-component (and a relation with $L(\pi_1 \times \pi_2,0)$). This conjecture is known: for $B=M_2(\mathbb{Q})$, it was confirmed in Beilinson's original paper \cite{Beilinson}; for $B$ non-split, a proof was given in \cite{RamakrishanShimuraCurves}. Note however that the proof in the latter case does not give any information about the cycle $Z$, and in fact \cite{RamakrishnanAnnArbor} asks if it is possible to use meromorphic functions whose divisor is supported on CM points to construct interesting higher Chow cycles.

The dichotomy described in the previous paragraph is related to the behaviour of $\pi_{1,2}:=\pi_1 \boxtimes \pi_2$ under the theta correspondence. Namely, one can consider $\pi_{1,2}$ as a representation of a general orthogonal group $GSO(V)$. For a cuspidal automorphic representation $\pi \subset \mathcal{A}_0(GSp_{2r})$, denote by $\theta(\pi) \subset \mathcal{A}(GSO(V))$ (the restriction to $GSO(V)(\mathbb{A})$ of) its global theta lift. It turns out that for our representations $\pi_1,\pi_2$ we have:
\begin{equation}
\begin{split}
\pi_1 \cong \pi_2^\vee & \Leftrightarrow \pi_{1,2} \subset \theta(\pi), \quad \pi \subset \mathcal{A}_0(GL_2) \\
\pi_1 \ncong \pi_2^\vee & \Leftrightarrow \pi_{1,2} \subset \theta(\pi), \quad \pi \subset \mathcal{A}_0(GSp_4).
\end{split}
\end{equation}
(The first result is due to Shimizu \cite{Shimizu}, while the second one is due to Roberts \cite{RobertsGlobalLpackets}.) Our second main result will be a more precise version of Theorem 1.1 in this context. Namely, fix automorphic representations $\pi_1$ and $\pi_2$ of $PB^\times$ corresponding to holomorphic forms of weight $2$ and full level on the Shimura curve $X_0^B$. Assume that $\pi_1 \ncong \pi_2^\vee$ and denote by $\Lambda(\pi_1 \times \pi_2,s)$ the complete Rankin-Selberg $L$-function attached to $\pi_1$ and $\pi_2$; note that under this assumption this $L$-function is regular at $s=0$.  Let $\Pi \subset \mathcal{A}_0(GSp_4)$ be the automorphic representation of $GSp_4(\mathbb{A})$ such that $\pi_{1,2} \subset \theta(\Pi)$ as above. In \autoref{section:Example_Products_Shimura_curves}, we will specify a newform $f \in \Pi$ and a Schwartz function $\varphi \in \mathcal{S}(V(\mathbb{A}_f^2))$, and we will prove the following result.

\begin{theorem} \label{thm:main_thm_shimura_curves}
Let $f \in \Pi$ be given by \eqref{eq:def_f} and $\varphi=\otimes_{v \nmid \infty } \varphi_v \in \mathcal{S}(V(\mathbb{A}^\infty)^2)$ with $\varphi_v$ as in \eqref{eq:def_varphi_example}. Then the form $\theta(f,\varphi \otimes \varphi_\infty) \in \mathcal{A}^{1,1}(X_K)$ is not identically zero. For any $T \in Sym_2(\mathbb{Q})_{>0}$, we have
\begin{equation*}
\begin{split}
([\tilde{\Phi}(T,\varphi,s')],\overline{\theta(f,\varphi \otimes \varphi_\infty)}) & =C \cdot (\widetilde{\mathcal{M}}_T(s'),f)^{reg} \\
& \quad \cdot \Lambda(\pi_1 \times \pi_2,0) \cdot \prod_v Z^0_v(f^\vee_v,f_v,\delta_*\Phi,1,-1/2),
\end{split}
\end{equation*}
where $C$ is a positive constant depending only on $B$, the zeta integrals $Z^0_v(f^\vee_v,f_v,\delta_*\Phi,1,-1/2)$ are non-vanishing and for all but finitely many $v$ we have $Z^0_v(f^\vee_v,f_v,\delta_*\Phi,1,-1/2)=1$.
\end{theorem}

\subsection{Outline of the paper} We now describe the contents of each section in more detail. \autoref{section:Pairing_differential forms} is devoted to the statement and proof of \autoref{thm:Rallis_pairing}. We start by recalling some basic facts about the Weil representation in Section \ref{subsection:Weil_representation}. In Section \ref{subsection:Local_zeta_integrals}, after reviewing the basic theory of local doubling zeta integrals for the group $Sp_{2r}$, we generalize some  of the results in \cite{PSRNew} on unramified representations to allow for ramified ones under a certain multiplicity one hypothesis. After stating the case of the Siegel-Weil formula that we need in Section \ref{subsection:Siegel_Weil_formla}, we prove in Section \ref{subsection:Euler_factorization} a result providing a factorization of a certain global integral into an Euler product. Section \ref{subsection:reps_and_cohomology} discusses some facts about the behaviour of representations with non-zero cohomology under the theta correspondence for real groups. Next, in Section \ref{subsection:K_infty_type_varphi}, we review some of the definitions and results in Part I and we prove some technical results about the behaviour of the Schwartz function $\tilde{\varphi}_\infty$ under the maximal compact subgroup $K_\infty$ of $Sp_4(\mathbb{R})$. Section \ref{subsection:theta_lifts_diff_forms} reviews the definition of theta lifts valued in differential forms.  With the results of Sections \ref{subsection:Local_zeta_integrals}-\ref{subsection:theta_lifts_diff_forms} as the main input, we state \autoref{thm:Rallis_pairing} in Section \ref{subsection:pairing_currents_forms} and give its proof in Section \ref{subsection:proof_thm_Rallis_pairing}.

The example of a product of Shimura curves described above is considered in \autoref{section:Example_Products_Shimura_curves}. The main goal is the proof of \autoref{thm:main_thm_shimura_curves}. We review some known results on the theta correspondence for the dual pair $(GSp_4,GO(V))$ in Section \ref{subsection:theta_corr_GSp4}. Once these results have been recalled, Section \ref{subsection:local_data} explains how to choose local data (local Schwartz functions and local vectors in certain automorphic representations of $GSp_4(\mathbb{A}_F)$) such that the global theta lift does not vanish and the hypotheses for Theorem \ref{thm:Rallis_pairing} are satisfied, leading to the proof of Theorem \ref{thm:main_thm_shimura_curves}. 

\subsection{Notation} \label{subsection:notation} The following conventions will be used throughout the paper.

\begin{itemize}
\item We write $\hat{\mathbb{Z}}=\varprojlim_{n}(\mathbb{Z}/n\mathbb{Z})$ and $\hat{M}=M \otimes_\mathbb{Z} \hat{\mathbb{Z}}$ for any abelian group $M$. We write $\mathbb{A}_f=\mathbb{Q} \otimes_\mathbb{Z} \hat{\mathbb{Z}}$ for the finite adeles of $\mathbb{Q}$ and $\mathbb{A}=\mathbb{A}_f \times \mathbb{R}$ for the full ring of adeles. 

\item For a number field $F$, we write $\mathbb{A}_F= F \otimes_\mathbb{Q} \mathbb{A}$, $\mathbb{A}_{F,f}=F \otimes_\mathbb{Q} \mathbb{A}_f$ and $F_\infty=F \otimes_\mathbb{Q} \mathbb{R}$. We will suppress $F$ from the notation if no ambiguity can arise.

\item For a finite set of places $S$ of $F$, we will denote by $\mathbb{A}_S$ (resp. $\mathbb{A}^S$) the subset of adeles in $\mathbb{A}_F$ supported on $S$ (resp. away from $S$).

\item We denote by $\psi_\mathbb{Q}=\otimes_v \psi_{\mathbb{Q}_v}:\mathbb{Q} \backslash \mathbb{A}_\mathbb{Q} \rightarrow \mathbb{C}^\times$ the standard additive character of $\mathbb{A}_\mathbb{Q}$, defined by
\begin{equation*}
\begin{split}
\psi_{\mathbb{Q}_p}(x)&=e^{-2\pi i x}, \text{ for } x \in \mathbb{Z}[p^{-1}]; \\ 
\psi_{\mathbb{R}}(x)&=e^{2\pi i x}, \text{ for } x \in \mathbb{R}.
\end{split}
\end{equation*}
If $F_v$ is a finite extension of $\mathbb{Q}_v$, we set $\psi_v=\psi_{\mathbb{Q}_v}(tr(x))$, where $tr:F_v \rightarrow \mathbb{Q}_v$ is the trace map. For a number field $F$, we write $\psi=\otimes_v \psi_v: F \backslash \mathbb{A}_F \rightarrow \mathbb{C}^\times$ for the resulting additive character of $\mathbb{A}_F$.

\item For a locally compact, totally disconnected topological space $X$, the symbol $\mathcal{S}(X)$ denotes the Schwartz space of locally constant, compactly supported functions on $X$. For $X$ a finite dimensional vector space over $\mathbb{R}$, the symbol $\mathcal{S}(X)$ denotes the Schwartz space of all $\mathcal{C}^\infty$ functions on $X$ all whose derivatives are rapidly decreasing.

\item For a ring $R$, we denote by $Mat_n(R)$ the set of all $n$-by-$n$ matrices with entries in $R$. The symbol $1_n$ (resp. $0_n$) denotes the identity (resp. zero) matrix in $Mat_n(R)$.

\item For a matrix $x \in Mat_n(R)$, the symbol $^t x$ denotes the transpose of $x$. We denote by $Sym_n(R)=\{x \in Mat_n(R)|x=^t x \}$ the set of all symmetric matrices in $Mat_n(R)$.

\item The symbol $X \coprod Y$ denotes the disjoint union of $X$ and $Y$.

\item If an object $\phi(s)$ depends on a complex parameter $s$ and is meromorphic in $s$, we denote by $CT_{s=s_0}\phi(s)$ the constant term of its Laurent expansion at $s=s_0$.

\end{itemize}

\subsection{Acknowledgments} Most of the work on this paper was done during my Ph.D. at Columbia University, and this work is part of my Ph.D. thesis. I would like to express my deep gratitude to my advisor Shou-Wu Zhang, for introducing me to this area of mathematics, for his guidance and encouragement and for many very helpful suggestions. I would also like to thank Stephen S. Kudla for answering my questions about the theta correspondence and for several very inspiring remarks and conversations. This paper has also benefitted from comments and discussions with Patrick Gallagher, Yifeng Liu, Andr\'e Neves, Ambrus P\'al, Yiannis Sakellaridis and Wei Zhang; I am grateful to all of them.

\section{Pairing currents and forms} \label{section:Pairing_differential forms}
The main goal of this section is to compute the values of the currents $[\tilde{\Phi}(T,\varphi,s)]$ and $[\Phi(T,\varphi)]-[\Phi(T^\iota,\varphi^\iota)]$ introduced in \cite{GarciaRegularizedLiftsI} on differential forms $\alpha$ arising as theta lifts from cusp forms on $Sp_4(\mathbb{A}_F)$. This is the content of \autoref{thm:Rallis_pairing} and \autoref{cor:thm_Rallis_pairing}, which prove a formula for these values that is analogous to the classical Rallis inner product formula for Petersson norms of theta lifts. Underlying Theorem \ref*{thm:Rallis_pairing} is a factorization of a certain global functional into a product of local ones; these functionals are defined in \autoref{subsection:Euler_factorization} and the Euler product factorization is the content of \autoref{proposition:Euler_product}. We will only need this factorization for the symplectic group of rank $2$, but we have chosen to prove it for general symplectic groups since the proof is essentially the same. The proof of \autoref{thm:Rallis_pairing} will be given in \autoref{subsection:proof_thm_Rallis_pairing}; the other subsections recall the necessary background needed for the statement and proof.

Recall that the space $V$ is assumed to have even dimension over $F$. From now on, we will also assume that $V$ is anisotropic over $F$, so that $X_K$ will be compact. This is always true if $F \neq \mathbb{Q}$ due to our assumption on the signature of $V$.

\subsection{Weil representation} \label{subsection:Weil_representation} Let $k$ be a local field of characteristic $0$ and let $(V,q)$ be a non-degenerate quadratic vector over $k$ of even dimension $m$. Denote by $\det(V)$ the element of $k^\times/(k^{\times})^2$ determined by $\det((v_i,v_j))$ for any $k$-basis $\{v_1,\ldots,v_{m}\}$ of $V$. Let $\chi_V$ be the quadratic character defined by
\begin{equation}
\chi_V(a)=(\det(a),(-1)^\frac{m}{2}\det(V))
\end{equation}
where $a \in GL_n(k)$ and $(\cdot,\cdot)$ denotes the Hilbert symbol.

Consider the space of Schwartz functions $\mathcal{S}(V^r)$ for $r \geq 1$ (see \autoref{subsection:notation}). It carries an action $\omega=\omega_\psi$ of the group $Sp_{2r}(k) \times O(V)$ known as the Weil representation. It depends on an additive character $\psi:k \rightarrow \mathbb{C}^\times$. For $\varphi \in \mathcal{S}(V^r)$ and $h \in O(V)$, this action is given by
\begin{equation}
\omega(1,h)\varphi(v)=\varphi(h^{-1}v).
\end{equation}
The action of $Sp_{2r}(k)$ is determined by the following expressions:
\begin{equation}
\begin{split}
\omega\left( \left(\begin{array}{cc} a & \\ & {}^t a^{-1} \end{array}\right),1 \right)\varphi(v) &=\chi_V(a) |\det(a)|^{\frac{m}{2}}\varphi(v \cdot a), \ \ \ a \in GL_r(k), \\
\omega\left( \left(\begin{array}{cc} 1_r & b \\ & 1_r \end{array}\right),1 \right)\varphi(v) &=\psi(tr(bq(v))) \varphi(v), \ \ \ b \in Sym_r(k), \\
\omega\left( \left(\begin{array}{cc}  & -1_r \\ 1_r &  \end{array}\right),1 \right)\varphi(v) &=\gamma_V^{-r} \int_{V^r}\varphi(w)\psi(-tr(v,w))dw.
\end{split}
\end{equation}
Here $dw$ is the self-dual Haar measure on $V^r$ with respect to $\psi$, and $\gamma_V$ is an explicit $8$-th root of unity (see e.g. \cite[Lemma A.1]{IchinoPullbacks} for its explicit form).

\subsection{Local zeta integrals for $Sp_{2r}$} \label{subsection:Local_zeta_integrals} \label{subsection:local_doubling_integrals} We review some well known facts about the local zeta integrals appearing in the doubling method for symplectic groups.

Let $(W,(\cdot,\cdot))$ be a symplectic vector space of dimension $2r$ over a local field $F$ and endow $\tilde{W}=W \oplus W$ with the symplectic form $(\cdot,\cdot) \oplus -(\cdot,\cdot)$. Let $G=Sp(W)$, $\tilde{G}=Sp(\tilde{W})$ and $\tilde{K}$ be a maximal compact subgroup of $\tilde{G}$. There is a natural map
\begin{equation}
\iota: G \times G \rightarrow \tilde{G}.
\end{equation}
This embedding can be described explicitly as follows. Denote by $Sp_{2r}$ the symplectic group of rank $r$ defined by
\begin{equation} \label{eq:def_Sp_2r}
Sp_{2r}=\{g \in GL_{2r} | {}^tg J g =J  \},
\end{equation}
where $J=\left( \begin{smallmatrix}  & -1_r \\ 1_r &  \end{smallmatrix} \right)$. Consider the embedding $\iota_0:Sp_{2r} \times Sp_{2r} \rightarrow Sp_{4r}$ given by
\begin{equation} \label{eq:def_iota_0}
\iota_0 \left(\begin{pmatrix}a & b \\ c & d \end{pmatrix},\begin{pmatrix}a' & b' \\ c' & d' \end{pmatrix}\right)=\begin{pmatrix}a & & b & \\  & a' & & b' \\ c & & d & \\ & c' & & d' \end{pmatrix}.
\end{equation}
Then we can choose a basis of $W$ such that $\iota$ is identified with
\begin{equation} \label{eq:def_iota}
\iota(g_1,g_2)=\iota_0\left(g_1,\begin{pmatrix}1_r &  \\  & -1_r \end{pmatrix}g_2 \begin{pmatrix}1_r &  \\  & -1_r \end{pmatrix}\right).
\end{equation} 

The subspace $W^d=\{(w,w)|w \in W \}$ is maximal isotropic in $\tilde{W}$; denote by $\tilde{P}$ its stabilizer in $\tilde{G}$. Then $\tilde{P}$ is a maximal parabolic subgroup of $\tilde{G}$ and $(G \times G) \cap \tilde{P} = \{ (g,g) | g \in G \}$. The assignment $p \mapsto \det(p_{W^d})$ defines a homomorphism $X:\tilde{P} \rightarrow F^\times$. For a quasi-character $\omega: F^\times \rightarrow \mathbb{C}^\times$ and $s \in \mathbb{C}$ we obtain a character $\omega |\cdot|^s: \tilde{P} \rightarrow \mathbb{C}^\times$ sending $p$ to $\omega(X(p)) |X(p)|^s$. Denote by $Ind_{\tilde{P}}^{\tilde{G}}(\omega |\cdot|^s)$ the normalized induction: elements of $Ind_{\tilde{P}}^{\tilde{G}}(\omega |\cdot|^s)$ are functions $f(g,s)$ on $\tilde{G}$ that satisfy
\begin{equation}
f(pg,s)=\omega(X(p)) \cdot |X(p)|^{s+dim(W)+\frac{1}{2}} f(g,s).
\end{equation}
We say that a function $f(g,s)$ on $\tilde{G} \times \mathbb{C}$ is a holomorphic section of $Ind_{\tilde{P}}^{\tilde{G}}(\omega |\cdot|^s)$ if $f(g,s)$ is right $\tilde{K}$-finite, holomorphic as a function of $s \in \mathbb{C}$ for any fixed $g \in G$ and belongs to $Ind_{\tilde{P}}^{\tilde{G}}(\omega |\cdot|^s)$ for any fixed $s \in \mathbb{C}$. A holomorphic section is called standard if its restriction to $\tilde{K} \times \mathbb{C}$ is independent of $s$.

Let $\pi$ be an irreducible admissible representation of $G$. Denote by $\pi^\vee$ its admissible dual and by $(\cdot,\cdot):\pi \otimes \pi^\vee \rightarrow \mathbb{C}$ the natural bilinear pairing. For $F$ archimedean, an admissible representation of $G$ will by definition be an admissible $(\mathfrak{sp}_{2r},U(r))$-module, with $\pi^\vee$ the contragredient $(\mathfrak{sp}_{2r},U(r))$-module. Note that in this case, even though $\pi$ is not a representation of $G$, one can still define matrix coefficients as functions on $G$.

For $v \in \pi$, $v^\vee \in \pi^\vee$ and $\Phi\in Ind_{\tilde{P}}^{\tilde{G}}(\omega |\cdot|^s)$, define
\begin{equation}
Z(v,v^\vee,\Phi,\omega,s)=\int_G (\pi^\vee(g)v^\vee,v) \cdot \Phi(\iota(1,g),s)dg.
\end{equation} 

The following theorem summarizes some fundamental properties of these zeta integrals.

\begin{theorem}[\cite{PSRepsilon_factors}, \cite{LapidRallis}] \label{thm:main_thm_local_zeta_int} Let $\pi$ be an irreducible admissible representation of $G$.
\begin{enumerate}
\item There exists $s_0 \in \mathbb{R}$ such that the integrals $Z(v,v^\vee,\Phi,\omega,s)$ with $\Phi$ a holomorphic section of $Ind_{\tilde{P}}^{\tilde{G}}(\omega |\cdot|^s)$ converge absolutely when $Re(s)>s_0$.
\item The integrals $Z(v,v^\vee,\Phi,\omega,s)$ admit meromorphic continuation to $s \in \mathbb{C}$. For standard $\Phi$ and non-archimedean $F$ with residue field of order $q$, they are rational functions in $q^{-s}$ with bounded denominators. For archimedean $F$ and standard $\Phi$, they are polynomial multiples of a fixed meromorphic function of $s$.
\item Assume that $F$ is non-archimedean. Let $\omega$ be an unramified character of $F^\times$ and $\pi$ be an admissible irreducible representation of $G$ that is spherical with respect to $K=Sp_{2r}(\mathcal{O}_v)$. Let $std: (Sp_{2r} \times GL_{1,F})^\vee = SO_{2r+1}(\mathbb{C}) \times \mathbb{C}^\times \rightarrow GL_{2r+1}(\mathbb{C})$ be the standard representation. Let $v_0 \in \pi$ and $v_0^\vee \in \pi^\vee$ be vectors fixed by $K$ such that $(v_0,v_0^\vee)=1$. Let $\Phi^0 \in Ind_{\tilde{P}}^{\tilde{G}}(|\cdot|^s)$ be the unique $Sp_{4r}(\mathcal{O})$-fixed vector such that $\Phi^0(1,s)=1$. Then
\[
Z(v_0,v_0^\vee,\Phi^0,\omega,s)=\frac{L(\pi \times \omega,std,s+\frac{1}{2})}{d(\omega,s)},
\]
where $L(\pi \times \omega,std,s)$ denotes the local Langlands $L$-factor and
\[
d(\omega,s)=L(\omega,s+r+1/2) \cdot \prod_{i=1}^r L(\omega^2,2s+2i-1)
\]
with $L(\omega,s)$ the Tate local $L$-factor of $\omega$.
\end{enumerate}
\end{theorem}

Part $(2)$ of this theorem makes it possible to introduce local $L$-factors $L_{PSR}(\pi,\omega,s)$ for any irreducible admissible representation of $Sp_{2r}$ (in fact of any classical group) over a non-archimedean local field $F$ and any character $\omega:F^\times \rightarrow \mathbb{C}^\times$. Namely, \cite{PSRTriple} and \cite{HarrisKudlaSweet} introduce a family of good sections $\Phi \in Ind_{\tilde{P}}^{\tilde{G}}(\omega |\cdot|^s)$ containing the standard sections. The set
\begin{equation}
\{Z(v,v^\vee,\Phi,\omega,s) \ | \ v \in \pi, v^\vee \in \pi^\vee, \Phi \text{ good} \}
\end{equation}
is a fractional ideal for the ring $\mathbb{C}[q^s,q^{-s}]$ of the form $(P(q^{-s})^{-1})$ for a unique polynomial $P(q^{-s})$ such that $P(0)=1$. Define
\begin{equation}
L_{PSR}(\pi,\omega,s+1/2)=\frac{1}{P(q^{-s})}.
\end{equation}

A different approach to defining local $L$-factors is proposed in \cite{LapidRallis}, where the study of the gamma factors arising in the functional equations of the zeta integrals above leads to an $L$-factor $L_{LR}(\pi,\omega,s)$. Both types of $L$-factors have been shown to agree in \cite{Yamana}. Thus we will simply denote this local $L$-factor by $L(\pi,\omega,s)$. For $\pi$ unramified and $\omega \equiv 1$, we have $L(\pi,\omega,s)=L(\pi,std,s)$.


With $d(\omega,s)$ as in the theorem above, define a normalized zeta integral by
\begin{equation} \label{eq:normalized_zeta_int}
Z^0(v,v^\vee,\Phi,\omega,s)=\left( \frac{L(\pi,\omega,s+1/2)}{d(\omega,s)} \right)^{-1}\cdot Z(v,v^\vee,\Phi,\omega,s).
\end{equation}
Then $Z^0(v,v^\vee,\Phi,\omega,s)$ is a meromorphic function of $s$ (holomorphic outside the set of poles of $d(\omega,s)$) provided that $\Phi$ is a good section.

The following lemma will be used in the proof of \autoref{thm:Rallis_pairing} to relate the zeta integrals appearing in the computation to the doubling zeta integrals just discussed. It generalizes \cite[Lemma 1]{PSRNew} (the case where $G=Sp_{2r}$, $K$ is maximal and $\sigma$ is the trivial representation). 
\begin{lemma}
Let $F$ be a local field and $G=\underline{G}(F)$ the set of $F$-points of a reductive algebraic group $\underline{G}$ over $F$. Let $K$ be an open compact subgroup of $G$ if $F$ is non-archimedean, and a maximal compact subgroup of $G$ if $F$ is archimedean. Let $\pi$ be an admissible, irreducible representation of $G$. If $F$ is archimedean, we assume that $\pi$ is a smooth irreducible Fr\'echet representation of moderate growth whose underlying $(\mathfrak{g},K)$-module $\pi_K$ is admissible.  Let $\sigma$ be an irreducible representation of $K$ and let $v_\sigma \in \sigma$, $v_\sigma \neq 0$. Assume that $\sigma$ appears with multiplicity one in $\pi$ (resp. $\pi_K$, for $F$ archimedean). Fix a $K$-invariant embedding $\iota_\pi: \sigma \hookrightarrow \pi$ and choose $v^\vee_{\sigma} \in \pi^\vee$ such that $\langle \pi^\vee(k)v^\vee_{\sigma} | k \in K \rangle \cong \sigma^{\vee}$ and $(\iota_\pi(v_\sigma),v^\vee_{\sigma})=1$.

Let $l:\pi \rightarrow \mathbb{C}$ be a linear functional. If $F$ is archimedean, assume that $l$ is continuous. Then, for all $v \in \pi$ and $\alpha \in \sigma^\vee$: 
\[
\int_{K} l(\pi(k)v) \cdot \alpha(\sigma(k^{-1})v_\sigma)dk=l(\iota_\pi (v_\sigma)) \cdot \int_{K} (\pi(k)v,v^\vee_{\sigma}) \cdot \alpha(\sigma(k^{-1})v_\sigma)dk.
\]

\end{lemma}

\begin{proof}
For any $v \in \pi$ and $\alpha \in \sigma^\vee$, define
\begin{equation*}
\begin{split}
h_1(v \otimes \alpha) &= \int_K l(\pi(k)v) \cdot \alpha(\sigma(k^{-1})v_\sigma) dk, \\
h_2(v \otimes \alpha) &= \int_K (\pi(k)v,v_{\sigma^\vee}) \cdot \alpha(\sigma(k^{-1})v_\sigma) dk.
\end{split}
\end{equation*}
Then $h_1, h_2 \in Hom_K(\pi \otimes \sigma^\vee,\mathbb{C})$ and are therefore proportional to each other (for $F$ archime\-dean, use that $h_1$ and $h_2$ are continuous, and the space of continuous $K$-invariant linear maps from $\pi$ to $\sigma$ embeds in $Hom_K(\pi_K,\sigma)$, which is one-dimensional). Evaluating them at $\iota_\pi v_\sigma \otimes \iota^\vee_\pi v_\sigma^\vee$ using the Schur orthogonality relations shows that
\[
h_1=l(\iota_\pi(v_\sigma)) \cdot h_2
\]
as required.
\end{proof}

\begin{corollary} \label{cor:evaluation_nonunique_local_zeta_int}
Let $G=Sp_{2r}(F)$ and let $\Phi \in Ind_{\tilde{P}}^{\tilde{G}}(\omega |\cdot|^s)$. Assume that, for fixed $g$ and $s$, the function $\Phi(\iota(1,kg),s)=\Phi(\iota(k^{-1},g),s):K \rightarrow \mathbb{C}$ is a matrix coefficient of $\sigma^\vee$, of the form $\alpha_{g,s}(\sigma(k)^{-1}v_\sigma)$. Then
\[
\int_{G} l(\pi(g)v) \cdot \Phi(\iota(1,g),s)dg=l(\iota_\pi (v_\sigma)) \cdot Z(v^\vee_{\sigma},v,\Phi,\omega,s).
\]
\end{corollary} 
\begin{proof} Replace the integrand in the left hand side with its average over $K$ and apply the previous lemma. 
\end{proof}

\subsection{A Siegel-Weil formula} \label{subsection:Siegel_Weil_formla} We will need a generalization of the original Siegel-Weil theorem in \cite{Weil} due to Kudla and Rallis. We recall its statement next.

Denote by $Sp_{2r}$ the symplectic group of rank $r$ defined by \eqref{eq:def_Sp_2r}. Let
\begin{equation}
\begin{split}
N_{2r}&=\left\{ n=n(X)= \left(\begin{array}{cc} 1_r & X \\ & 1_r \end{array} \right) | X= {}^tX \in Sym_r \right\}, \\
M_{2r}&=\left\{ m=m(a)= \left(\begin{array}{cc} a &  \\ & {}^t a^{-1} \end{array} \right) | a \in GL_r \right\}, \\
P_{2r}&=N_{2r} M_{2r}.
\end{split}
\end{equation}
Then $P=P_{2r}$ is a maximal parabolic subgroup of $Sp_{2r}$ known as the Siegel parabolic. Let $F$ be a totally real number field and define a maximal compact subgroup $K=\Pi_v K_v$ of $Sp_{2r}(\mathbb{A}_F)$ with $K_v=Sp_{2r}(\mathcal{O}_v)$ (resp. $K_v \cong U(r)$) when $v$ is non-archimedean (resp. archimedean). For any place $v$ of $F$, one has $Sp_{2r}(F_v)=N_{2r}(F_v)M_{2r}(F_v)K_v$ and hence
\begin{equation}
|a(g)|=|\det(a)|
\end{equation}
for $g=n(X)m(a)k$ is a well-defined quantity.

Let $V$ be a non-degenerate quadratic vector space over $F$ of even dimension $m$. Denote by $\det(V)$ the element of $F^\times/(F^{\times})^2$ determined by $\det((v_i,v_j))$ for any $F$-basis $\{v_1,\ldots,v_{m}\}$ of $V$. Let $\chi_V$ be the quadratic character defined by
\begin{equation}
\chi_V(a)=(\det(a),(-1)^\frac{m}{2}\det(V))_{\mathbb{A}}
\end{equation}
where $a \in GL_r(F)$ and $(\cdot,\cdot)_\mathbb{A}=\Pi_v (\cdot,\cdot)_v$ denotes the Hilbert symbol. For a Schwartz function $\varphi \in \mathcal{S}(V(\mathbb{A})^r)$, $g \in Sp_{2r}(\mathbb{A}_F)$ and $s \in \mathbb{C}$, define
\begin{equation}
\Phi(g,s)=\omega(g)\varphi(0) \cdot |a(g)|^{s-\frac{m-r-1}{2}}.
\end{equation}
Then
\begin{equation}
\Phi(nm(a)g,s)=\chi_V(a) |a|^{s+\frac{r+1}{2}} \Phi(g,s),
\end{equation}
that is, $\Phi(g,s)$ defines an element of the induced representation $Ind_{P(\mathbb{A}_F)}^{Sp_{2r}(\mathbb{A}_F)}(\chi_V |\cdot|^s)$. Form the Eisenstein series attached to $\Phi$:
\begin{equation} \label{eq:def_Eisenstein_series}
E(g,\Phi,s)=\sum_{\gamma \in P(F)\backslash Sp_{2r}(F)}\Phi(\gamma g,s).
\end{equation}
The sum converges absolutely for $Re(s)>(r+1)/2$ and the function $E(g,\Phi,s)$ admits meromorphic continuation to $s \in \mathbb{C}$.
\begin{theorem}\cite{KudlaRallis1}\label{thm:SiegelWeil}
Assume that $V$ is anisotropic over $F$. Then:
\begin{enumerate}
\item[i)] The function $E(g,\Phi,s)$ is holomorphic at $s_0=(m-r-1)/2$.
\item[ii)] Let
\begin{equation*} \label{eq:def_kappa_m}
\kappa_m = \left\{ \begin{array}{cc} 1, & \text{ if } m >r+1 \\ 
2, & \text{ if } m\leq r+1. \end{array} \right.
\end{equation*}
Then
\begin{equation*}
E(g,\Phi,s_0)=\kappa_m \int_{O(V)(F) \backslash O(V)(\mathbb{A})} \theta(g,h;\varphi)dh,
\end{equation*}
where $dh$ is an invariant measure normalized so that 
\[
Vol(O(V)(F) \backslash O(V)(\mathbb{A}),dh)=1.
\]
\end{enumerate}
\end{theorem}

\subsection{Euler factorization of a global integral} \label{subsection:Euler_factorization} In this section we introduce a global integral over the automorphic quotient of $Sp_{2r,F}$. This integral (for $r=2$) will appear in the proof of \autoref{thm:Rallis_pairing} concerning the values of the currents $[\tilde{\Phi}(T,\varphi,s)]$ introduced above. The main result will give an Euler product factorization for it, with local factors given by the local doubling zeta integrals reviewed in \autoref{subsection:Local_zeta_integrals}.

We fix some notation that will hold throughout this section. We write $G=Sp_{2r,F}$. Let $K^\infty=\Pi_{v \nmid \infty} K_v$ be an open compact subgroup of $G(\mathbb{A}_F^\infty)$ and define $K=K^\infty \times \Pi_{v | \infty}U(r)$. Fix a Haar measure  $dk=\Pi_v dk_v$ on $K$ defined by local measures such that $Vol(K_v,dk_v)=1$ for every place $v$. Let $\sigma=\otimes'_{v} \sigma_v$ be an irreducible representation of $K$. Thus $\sigma$ is finite dimensional and for $v$ outside some finite set of places, the local component $\sigma_v$ is the trivial representation. Let $\pi=\otimes'_v \pi_v \subset \mathcal{A}_0(G)$ be a cuspidal automorphic representation of $G(\mathbb{A}_F)$. We assume that
\begin{equation} \label{eq:global_mult_one_assumption}
\dim_\mathbb{C} Hom_K(\sigma,\pi)=1
\end{equation}
and fix a (unique up to scalar multiplication) $K$-equivariant embedding $\iota_\pi=\otimes_v \iota_{\pi,v}:\sigma \hookrightarrow \pi$.

Write $\tilde{G}=Sp_{4r,F}$ and let $P=P_{4r} \subset \tilde{G}$ be its Siegel parabolic subgroup defined in \autoref{subsection:Siegel_Weil_formla}. Given $\Phi \in Ind_{P(\mathbb{A}_F)}^{\tilde{G}(\mathbb{A}_F)}(\chi_V|\cdot|^s)$ and vectors $w \in \sigma$ and $w^\vee \in \sigma^\vee$, define a section $\Phi_{w^\vee,w} \in Ind_{P(\mathbb{A}_F)}^{\tilde{G}(\mathbb{A}_F)}(\chi_V|\cdot|^s)$ by
\begin{equation}
\Phi_{w^\vee,w}(g,s)=\int_{K} (\sigma^\vee(k)w^\vee,w) \cdot \Phi(g\iota(k,1),s)dk.
\end{equation}
For every $k \in K$, we have
\begin{equation} \label{eq:Phi_w,wvee_invariance}
(r(\iota(k,1))\Phi)_{\sigma^\vee(k)w^\vee,w}=\Phi_{w^\vee,w}.
\end{equation}
and
\begin{equation} \label{eq:Phi_w,wvee_invariance_2}
\Phi_{w^\vee,w}(g\iota(k,1),s)=\Phi_{w^\vee,\sigma(k)w}(g,s).
\end{equation}
Note that if $w=\otimes_v w_v$, $w^\vee=\otimes_v w^\vee_v$ and $\Phi=\otimes_v \Phi_v$ are pure tensors, then
\begin{equation}
\Phi_{w^\vee,w}(g,s)=\prod_v (\Phi_v)_{w_v^\vee,w_v}(g_v,s),
\end{equation}
where we define
\begin{equation} \label{eq:def_Phi_v_wveew}
(\Phi_v)_{w_v^\vee,w_v}(g_v,s),=\int_{K_v} (\sigma_v^\vee(k)w^\vee_v,w_v) \cdot \Phi_v(g_v\iota(k_v,1),s)dk_v.
\end{equation}
If $v$ is non-archimedean, $\sigma_v$ is the trivial representation of $Sp_{4r}(\mathcal{O}_v)$ and $\Phi^0_v$ is $Sp_{4r}(\mathcal{O}_v)$-invariant, then
\begin{equation} \label{eq:proj_unramified_section}
(\Phi^0_v)_{w_v^\vee,w_v}=\Phi^0_v
\end{equation}
provided that $(w_v,w_v^\vee)=1$.

Denote by $E(g,\Phi,s)$ the Eisenstein series attached to $\Phi$ (see \eqref{eq:def_Eisenstein_series}) and let $f \in \pi$ be a cusp form. For $g' \in G(\mathbb{A}_F)$ and $w$, $w^\vee$ and $\Phi$ as above, the integral
\begin{equation}
Z_w(g';f,w^\vee,\Phi)=\int_{G(F) \backslash G(\mathbb{A}_F)} f(g) \cdot E(\iota(g',g),\Phi_{w^\vee,w},s)dg
\end{equation}
converges for $s$ outside the set of poles of the Eisenstein series due to the cuspidality of $f$. By \eqref{eq:Phi_w,wvee_invariance}, it defines a functional
\begin{equation}
Z_w(g'; \cdot) \in Hom_{K \times G(\mathbb{A}_F)}(\pi \otimes \sigma^\vee \otimes Ind_{P(\mathbb{A}_F)}^{\tilde{G}(\mathbb{A}_F)}(\chi_V|\cdot|^s),\mathbb{C})
\end{equation}
for $s$ outside the said finite set. Here we regard $Ind_{P(\mathbb{A}_F)}^{\tilde{G}(\mathbb{A}_F)}(\chi_V|\cdot|^s)$ as a representation of $K \times G(\mathbb{A}_F)$ through the embedding $\iota$ defined by \eqref{eq:def_iota}.

Let us now introduce local functionals at every place $v$ of $F$. These functionals are defined using the local zeta integrals reviewed in \autoref{subsection:Local_zeta_integrals}. In that section we used local sections of an induced representation $Ind_{\tilde{P}}^{\tilde{G}}$ with respect to a certain parabolic $\tilde{P}$ different from the Siegel parabolic $P$ of $\tilde{G}$. These parabolics are conjugate: the matrix
\begin{equation} \label{eq:def_delta}
\delta = \left( \begin{array}{cccc}  &  & 1_2 &  \\
 & 1_2 &  &  \\
-1_2 & 1_2 &  &  \\
 &  & 1_2 & 1_2
 \end{array} \right) \in \tilde{G}(F)
\end{equation}
satisfies $\delta \tilde{P} \delta^{-1} = P$. For any place $v$ of $F$ it induces an isomorphism of $\tilde{G}(F_v)$-representations 
\begin{equation}
\delta_*:Ind_{P(F_v)}^{\tilde{G}(F_v)}(\chi_{V,v} |\cdot|^s) \rightarrow Ind_{\tilde{P}(F_v)}^{\tilde{G}(F_v)}(\chi_{V,v} |\cdot|^s)
\end{equation}
given by
\begin{equation}
\delta_*\Phi(g,s)=\Phi(\delta g,s).
\end{equation}
Note that in particular we have 
\begin{equation} \label{eq:Phi_delta_inv}
\Phi(\delta \iota(g',g'g),s)=\Phi(\delta \iota (1,g),s).
\end{equation}

\begin{lemma} \label{lemma:int_independence}
Let $v$ be a place of $F$ and let $w_v \in \sigma_v$ and $f_v^\vee \in \pi_v^\vee$ be such that $(\iota_{\pi,v}(w_v),f_v^\vee)=0$. Then we have
\[
Z(f_v^\vee,f_v,\delta_*\Phi_{w_v^\vee,w_v},\chi_{V,v},s)=0
\]
for any $f_v \in \pi$, $w_v^\vee \in \sigma_v^\vee$ and $\Phi \in Ind_{P(F_v)}^{\tilde{G}(F_v)}(\chi_{V,v}|\cdot|^s)$.
\end{lemma}
\begin{proof}
It suffices to prove this identity for $Re(s) \gg 0$, where
\[
Z(f_v^\vee,f_v,\delta_*\Phi_{w_v^\vee,w_v},\chi_{V,v},s)=\int_{G(F_v)} (\pi(g)f_v,f_v^\vee) \cdot \Phi_{w_v^\vee,w_v}(\delta \iota(1,g),s)dg
\]
and the integral converges absolutely. Changing variables ($g=k^{-1} g'$) and integrating over $K_v$ shows that it suffices to prove that the integral
\[
\int_{K_v} (\pi(g)f_v,\pi^\vee(k) f_v^\vee) \cdot \Phi_{w_v^\vee,w_v}(\delta \iota(1,k^{-1}g),s)dk.
\]
vanishes for every $g \in G(F_v)$. By \eqref{eq:Phi_delta_inv} and \eqref{eq:Phi_w,wvee_invariance_2}, we have
\[
\Phi_{w_v^\vee,w_v}(\delta \iota(1,k^{-1}g),s)=\Phi_{w_v^\vee,w_v}(\delta\iota(k,g),s)=\Phi_{w_v^\vee,\sigma_v(k)w_v}(\delta\iota(1,g),s).
\]
For fixed $g$ and $s$, this expression is a matrix coefficient of $\sigma_v$, of the form $(\sigma_v(k)w_v,\alpha)$, with $\alpha=\alpha_{g,s} \in \sigma_v^\vee$. The result follows from Schur orthogonality.
\end{proof}

Let $v$ be a place of $F$ and let $f_v^\vee \in \pi_v^\vee$ and $w_v \in \sigma_v$ such that $(\iota_{\pi,v}(w_v),f_v^\vee)=1$. Then \autoref{lemma:int_independence} implies that the normalized zeta integral $Z^0(f_v^\vee,f_v,\delta_*\Phi_{w_v^\vee,w_v},\chi_{V,v},s)$ (for $s$ outside the set of poles of $d(\chi_{V,v},s)$) given by \eqref{eq:normalized_zeta_int} does not depend on the choice of $f_v^\vee$ and defines a functional
\begin{equation}
Z^0_{\iota_v(w_v)} \in Hom_{K_v \times G(F_v)}(\pi_v \otimes \sigma_v^\vee \otimes Ind_{P(F_v)}^{\tilde{G}(F_v)}(\chi_{V,v}|\cdot|^s),\mathbb{C}).
\end{equation}
Moreover, for $v$ outside a finite set of places $S$, the representation $\sigma_v$ is trivial and $\pi_v$ is unramified. For such $v$, if we assume that $(w_v,w_v^\vee)=1$, that the section $\delta_*\Phi_v$ equals the unramified section $\Phi^0_v$ described in \autoref{thm:main_thm_local_zeta_int}.(3) and that $f_v$, $f_v^\vee$ are unramified vectors such that $(f_v,f_v^\vee)=1$, then
\begin{equation}
Z^0(f_v^\vee,f_v,\delta_*\Phi_{w_v^\vee,w_v},\chi_{V,v},s)=1
\end{equation}
by \eqref{eq:proj_unramified_section}. It follows that for any $w=\otimes_v w_v \in \sigma$, the product
\begin{equation}
Z^0_{\iota_\pi(w)}=\prod_v Z^0_{\iota_{\pi,v}(w_v)}
\end{equation}
defines a linear functional in $Hom_{K \times G(\mathbb{A})}(\pi \otimes \sigma^\vee \otimes Ind_{P(\mathbb{A})}^{\tilde{G}(\mathbb{A})}(\chi_V|\cdot|^s),\mathbb{C})$.

Under our muliplicity one assumption, any other $K$-equivariant embedding  $\iota'_\pi:\sigma \hookrightarrow \pi$ is of the form $\iota'_\pi=c \cdot \iota_\pi$ for some $c \in \mathbb{C}^\times$. For such $\iota'_\pi$, we have
\begin{equation} \label{eq:Z_functional_invariance}
Z^0_{\iota'_\pi(w)}=c^{-1} \cdot Z^0_{\iota_\pi(w)}.
\end{equation}

To compare the global functionals $Z_w$ and $Z^0_{\iota_\pi(w)}$, we need to introduce a global $L$-function $\Lambda(\pi,\chi_V,s)$. Recall that in \autoref{subsection:Local_zeta_integrals} we have defined local $L$-factors $L(\pi_v,\chi_{V,v},s)$ and $d(\chi_{V,v},s)$ for each place $v$. Define the completed global $L$-function
\begin{equation} \label{eq:def_global_L}
\Lambda(\pi,\chi_V,s)=\prod_v L(\pi_v,\chi_{V,v},s).
\end{equation}
This product converges for $Re(s)>r+1$ and admits meromorphic continuation to $s \in \mathbb{C}$ and a functional equation (see \cite{PSRbook, KudlaRallisPoles}). 
Define also
\begin{equation} \label{eq:def_global_d}
d(\chi_V,s)=\prod_v d(\chi_{V,v},s).
\end{equation}
Denote by $\Lambda(\chi_V,s)$ the completed Hecke $L$-function of $\chi_V$ and by $\Lambda_F(s)$ the completed Dedekind zeta function of $F$. Since $\chi_V$ is a quadratic character, we have
\begin{equation} \label{eq:global_d_computation}
d(\chi_V,s)=\Lambda(\chi_V,s+r+1/2) \cdot \prod_{i=1}^r \Lambda_F(2s+2i-1).
\end{equation}

We can now state the main result of this section showing that the functionals $Z_w$ and $Z^0_{\iota_\pi(w)}$ are proportional. Note that by \eqref{eq:Z_functional_invariance}, both sides of the identity in the following Proposition are independent of the embedding $\iota_\pi$.

\begin{proposition} \label{proposition:Euler_product}
Assume that $\dim_\mathbb{C} Hom_K(\sigma,\pi)=1$ and choose a $K$-equivariant embedding $\iota_\pi=\otimes_v \iota_{\pi,v}:\sigma \hookrightarrow \pi$. Let $w=\otimes_v w_v \in \sigma$ and denote by $f_{\iota_\pi(w)}$ the automorphic form in $\pi$ corresponding to $\iota_\pi(w)$. For any $g' \in Sp_{2r}(\mathbb{A}_F)$ we have
\begin{equation}
Z_w(g')=f_{\iota_\pi(w)}(g') \cdot \frac{\Lambda(\pi^\vee,\chi_V,s+1/2)}{d(\chi_V,s)} \cdot Z^0_{\iota_\pi(w)}
\end{equation} 
as functionals in $Hom_{K \times G(\mathbb{A})}(\pi \otimes \sigma^\vee \otimes Ind_{P(\mathbb{A})}^{\tilde{G}(\mathbb{A})}(\chi_V,s),\mathbb{C})$.
\end{proposition}
\begin{proof}
It suffices to prove the identity for standard $\Phi$ and $Re(s)>(r+1)/2$, where the Eisenstein series $E(g,\Phi,s)$ is defined by the sum \eqref{eq:def_Eisenstein_series}. For such $s$, the evaluation of the left hand side proceeds, as in the usual doubling method, by analyzing the orbits of $G(F) \times G(F)$ on $P(F) \backslash \tilde{G}(F)$. One finds (see e.g. \cite[Prop. 2.1]{GelbartPSRallis}) that there are $r+1$ orbits and that, due to the cuspidality of $f$, only one of them gives a non-zero contribution. Moreover, let $\delta \in \tilde{G}(F)$ be as in \eqref{eq:def_delta}.
Then $\delta$ is a representative of the non-negligible orbit such that $Stab_{G \times G}(P(F)\delta)=\Delta G$, where $\Delta G=\{(g,g)| g \in G \} \subset G \times G$ (see \cite{KudlaRallis2005}). Unfolding the integral we obtain
\[
\int_{G(F) \backslash G(\mathbb{A}_F)} f(g) \cdot E(\iota(g',g),\Phi_{w^\vee,w},s)dg=\int_{G(\mathbb{A}_F)} f(g) \cdot \Phi_{w^\vee,w}(\delta \iota(g',g),s)dg.
\]
A change of variables ($g=g'g''$) shows that this equals
\begin{equation}\label{eq:proof_thm_3/2_eq1}
\int_{G(\mathbb{A}_F)} f(g'g) \cdot \Phi_{w^\vee,w}(\delta \iota(1,g),s)dg
\end{equation}
by \eqref{eq:Phi_delta_inv}. Let $f=\otimes_v f_v$ in $\pi$, $w^\vee=w_v^\vee \in \sigma^\vee$ and $\Phi=\otimes_v \Phi_v \in Ind_{P(\mathbb{A}_F)}^{\tilde{G}(\mathbb{A}_F)}(\chi_V |\cdot|^s)$ and choose $f^\vee=\otimes_v f_v^\vee \in \pi^\vee$ such that $(\iota_{\pi,v}(w_v),f_v^\vee)=1$ for every place $v$. To prove the claim, we need to show that
\begin{equation*}
\eqref{eq:proof_thm_3/2_eq1}=f_{\iota_\pi(w)}(g') \cdot \frac{\Lambda(\pi^\vee,\chi_V,s+1/2)}{d(\chi_V,s)} \cdot \prod_v Z^0(f_v^\vee,f_v,\delta_*\Phi_{w_v^\vee,w_v},\chi_{V,v},s).
\end{equation*}
%
Fix a finite set of places $S_0$ such that for $v \notin S_0$ all the data are unramified. More precisely, assume that for such $v$ the representation $\pi_v$ is unramified; the vector $f_v \in \pi_v$ is fixed by $G(\mathcal{O}_v)$; the representation $\sigma_v$ is trivial and we have $\iota_{\pi,v}(w_v)=f_v$ and $(w_v,w_v^\vee)=1$; $\Phi_v=\Phi^0_v$ with $\Phi^0_v$ as in \autoref{thm:main_thm_local_zeta_int}; and $\chi_{V,v}$ is unramified. For a finite set of places $S$ containing $S_0$, write
\begin{equation}
G(\mathbb{A}_S)=\prod_{v \in S} G(F_v) \times \prod_{v \notin S} G(\mathcal{O}_v).
\end{equation}
Then $G(\mathbb{A}_F) = \varinjlim_S G(\mathbb{A}_S)$ where the limit is over a family of $S \supset S_0$ whose union is the set of all places of $F$. Hence
\begin{equation*}
\int_{G(\mathbb{A}_F)} f(g'g) \Phi_{w^\vee,w}(\delta \iota(1,g),s) \ dg=\varinjlim_{S} \int_{G(\mathbb{A}_S)} f(g'g) \Phi_{w^\vee,w}(\delta \iota(1,g),s) \ dg
\end{equation*}
and it suffices to show that
\begin{equation*}
\int_{G(\mathbb{A}_S)} f(g'g) \Phi_{w^\vee,w}(\delta \iota(1,g),s) \ dg=f_{\iota_\pi(w)}(g') \cdot \prod_{v \in S} Z(f^\vee_v,f_v,\delta_*\Phi_{w_v^\vee,w_v},\chi_{V,v},s)
\end{equation*}
for $S \supset S_0$. For a place $v$ of $F$, write $G(\mathbb{A}^v_F)=\Pi_{w \neq v}' G(F_w)$, so that $G(\mathbb{A}_F)=G(F_v)\cdot G(\mathbb{A}^v_F)$. For fixed $g'$ and $g^v$ in $G(\mathbb{A}^v_F)$, we have $f(g'g^v g_v)=l(\pi(g_v)f_v)$ for some linear functional $l:\pi_v \rightarrow \mathbb{C}$; if $v$ is archimedean, then this functional $l$ is continuous for the Fr\'echet topology on $\pi_v$, since it is given by evaluation of a $\mathcal{C}^\infty$ automorphic form. Moreover, by \eqref{eq:Phi_delta_inv} and \eqref{eq:Phi_w,wvee_invariance_2} we have
\[
\Phi_{w_v^\vee,w_v}(\delta \iota(1,k_v g_v),s)=\Phi_{w_v^\vee,w_v}(\delta \iota(k_v^{-1},g_v),s)=\Phi_{w_v^\vee,\sigma_v(k_v)^{-1}w_v}(\delta \iota(1,g_v),s)
\]
and hence, for fixed $g_v$ and $s$, the section $\Phi_{w_v^\vee,w_v}(\delta(1,k_v g_v),s)$ is a matrix coefficient of $\sigma_v^\vee$, of the form $\alpha(\sigma_v(k_v)^{-1}w_v)$ with $\alpha=\alpha_{g_v,s} \in \sigma_v^\vee$. The result now follows from \autoref{cor:evaluation_nonunique_local_zeta_int} applied to each place $v \in S$.
%
%
\end{proof}

The statement of \autoref{thm:Rallis_pairing} will involve the value of the completed $L$-function appearing in the above proposition at a certain value $s_0$. The following proposition shows that for certain representations $\pi$ this value is finite.

\begin{proposition} \label{prop:completed_L_fcn_regular}
Let $\pi \subset \mathcal{A}_0(Sp_{4,F})$ be an irreducible cuspidal automorphic representation. Assume that there exists an anisotropic quadratic vector space $V$ over $F$ of even dimension $m \geq 4$ such that the global theta lift $\Theta(\pi)$ contains a non-zero cusp form on $O(V)(\mathbb{A})$. Then the meromorphic function
\begin{equation*}
\frac{\Lambda(\pi^\vee,\chi_V,s+1/2)}{d(\chi_V,s)}
\end{equation*}
is regular and non-vanishing at $s_0=(m-5)/2$.
\end{proposition}
\begin{proof}
By the main theorem in \cite{MoeglinInvolution}, the global theta lift $\tau=\Theta(\pi)$ is cuspidal and irreducible and satisfies $\pi=\Theta(\tau)$. Consider first the case $m=4$. Then $d(\chi_V,s)$ has a simple pole at $s_0=-1/2$. By the functional equation of $\Lambda(\pi^\vee,\chi_V,s+1/2)$ (see \cite[\textsection 10]{LapidRallis}), it suffices to show that $\Lambda(\pi,\chi_V,s)$ has a simple pole at $s=1$. Since $\Lambda(\pi,\chi_V,s)$ has at most simple poles (see \cite[Thm 9.1]{Yamana}), it suffices to show that the incomplete $L$-function $L^S(\pi,\chi_V,s)$ has a simple pole at $s=1$ for some finite set of places $S$. But the cuspidality of $\Theta(\tau)$ implies that $L^S(\tau,s)$ does not vanish at $s=1$ (see \cite[Thm. 2]{Yamana}) and hence $L^S(\pi,\chi_V,s)$ has a simple pole at $s=1$ by \cite[Cor. 7.1.5]{KudlaRallisSiegel1} as required.

Assume now that $m > 4$; in this case $d(\chi_V,s)$ is regular at $s_0$. We first prove that $\Lambda(\pi^\vee,\chi_V,s+1/2)$ has no pole at $s_0$. Note that $\Lambda(\pi^\vee,\chi_V,s)$ can only have poles at $s \in \{-1,0,1,2\}$ by \cite[Thm. 9.1]{Yamana}, hence we only need to consider $m \in \{6,8\}$. Suppose first that $m=8$; by the functional equation, it suffices to show that $\Lambda(\pi,\chi_V,s)$ is regular at $s=-1$. This follows from \cite[Thm 10.1]{Yamana}: if $\Lambda(\pi,\chi_V,s)$ had a pole at $s=-1$, then $V$ would be in the same Witt tower as a certain two-dimensional quadratic vector space over $F$; this is absurd since $V$ is anisotropic. The same argument works for $m=6$.

Finally, consider the non-vanishing at $s_0$. The Rallis inner product formula for $||\theta_\varphi(f)||^2$ with $f \in \overline{\pi}$ and $\varphi \in \mathcal{S}(V(\mathbb{A}^2))$ shows that there exists a finite set of places $S$ such that the incomplete function $d^S(\chi_V,s)^{-1}L^S(\pi^\vee,\chi_V,s+1/2)$ does not vanish at $s_0$ . The result follows since for any place $v$, the local functions $d(\chi_{V,v},s)^{-1}L(\pi_v^\vee,\chi_{V,v},s+1/2)$ do not vanish at $s_0$.
\end{proof}

\subsection{Representations of $SO(n,2)^+$ with non-vanishing cohomology} \label{subsection:reps_and_cohomology} Consider an irreducible unitary representation $(\pi,V_\pi)$ of $SO(n,2)^+$ and denote by $H^*_{cts}(SO(n,2)^+,V_\pi)$ the continuous cohomology of $SO(n,2)^+$ with coefficients in $V_\pi$ (see \cite{BorelWallach}). We say that it has non-vanishing cohomology if 
\begin{equation}
H^*_{cts}(SO(n,2)^+,V_\pi) \neq 0.
\end{equation}
The infinitesimal equivalence classes of such representations were classified in \cite{VoganZuckerman}. They are all of the form $A_\mathfrak{q}=A_{\mathfrak{q}}(0)$ where $\mathfrak{q}$ is a certain type of parabolic subalgebra of $\mathfrak{so}(n,2)_\mathbb{C}$ called a $\theta$-stable subalgebra. The algebra $\mathfrak{q}$ is determined by its normalizer $L$ in $SO(n,2)^+$, and the $(\mathfrak{so}(n,2)_\mathbb{C},K)$-module $A_\mathfrak{q}$ (here $K=SO(n) \times SO(2)$) is obtained by a process of cohomological induction from the trivial representation of $L$.

We are interested in representations that contribute non-trivial $(n-1,n-1)$-classes to the cohomology of $X_K$, that is, such that $H^{n-1,n-1}_{cts}(SO(n,2)^+,V_\pi) \neq 0$; we exclude the trivial representation. The results in \cite{VoganZuckerman} (see also \cite[\textsection 15.2]{Schwermer} for an explicit description in our case) show that for $n>2$ there is only one such representation up to infinitesimal equivalence. Its $(\mathfrak{so}(n,2)_\mathbb{C},K)$-module is of the form  $A_\mathfrak{q}$, with
\begin{equation}
L \cong U(1) \times SO(n-2,2)^+.
\end{equation}
Moreover, its Lie algebra cohomology is given by
\begin{equation}
H^{p,q}(\mathfrak{so}(n,2)_\mathbb{C},K,A_\mathfrak{q})= \left\{ \begin{array}{cc}  \mathbb{C}^2,& \text{ if } p=q=n, \\
\mathbb{C},& \text{ if } 0 < p=q < n \quad (p=q\neq n), \\
0,& \text{ otherwise.} \end{array}  \right.
\end{equation}
The representation $A_\mathfrak{q}$ extends to a representation of $SO(n,2)$ with non-vanishing $(\mathfrak{so}(n,2),S(O(n)\times O(2)))$ cohomology; we denote the extension still by $A_\mathfrak{q}$.

If $n=2$, there are two non-trivial representations $A_{\mathfrak{q}^+}$ and $A_{\mathfrak{q}^-}$ with non-vanishing $(1,1)$-cohomology; we have $L \cong U(1)\times U(1)$ and
\begin{equation}
H^{p,q}(\mathfrak{so}(n,2)_\mathbb{C},K,A_{\mathfrak{q}^{\pm}})= \left\{ \begin{array}{cc}  \mathbb{C},& \text{ if } p=q=1, \\
0,& \text{ otherwise.} \end{array}  \right.
\end{equation}
In fact this can be seen more directly using the special isomorphism
\begin{equation}
SO(2,2)^+ \cong \{\pm 1\} \backslash SL_2(\mathbb{R}) \times SL_2(\mathbb{R}).
\end{equation}
Namely, a $(\mathfrak{sl}_2,SO(2))$-module $\sigma$ associated with an irreducible unitary non-trivial representation of $SL_2(\mathbb{R})$ with non-vanishing cohomology has to be infinite-dimensional with zero eigenvalue for the Casimir element (see \cite[Thm. I.5.3]{BorelWallach}); thus $\sigma \in \{\mathcal{D}_2,\mathcal{D}_{-2}\}$, where $\mathcal{D}_k$ for positive (resp. negative) $k$ denotes the $(\mathfrak{sl}_2,SO(2))$-module with lowest (resp. highest) weight $k$ and trivial central character. We conclude from the K\"{u}nneth formula that the irreducible non-trivial representations with non-vanishing $(1,1)$-cohomology of $SO(2,2)^+$ are
\begin{equation}
\{A_{\mathfrak{q}^+},A_{\mathfrak{q}^-} \}=\{ \mathcal{D}_2 \boxtimes \mathcal{D}_{-2}, \mathcal{D}_{-2} \boxtimes \mathcal{D}_{2} \}.
\end{equation}

\subsubsection{Theta correspondence} The relationship between representations with non-vanishing cohomology and the theta correspondence (for type I dual pairs) was studied by Li \cite{LiCohomology}. His results show that the representations with non-vanishing $(n-1,n-1)$-cohomology discussed above all appear as restrictions to $SO(n,2)^+$ of theta lifts of discrete series on $SL_2(\mathbb{R})$.

\begin{proposition} \label{prop:Archimedean_theta_SL2} \cite[Thm. 6.2]{LiCohomology}
Let $n$ be a positive integer and let $\mathcal{D}_{n+1}$ be the discrete series representation of $SL_2(\mathbb{R})$ of lowest weight $n+1$. Consider the dual pair $(SL_2(\mathbb{R}),O(2n,2))$ and let $\pi=\theta(\mathcal{D}_{n+1})$ the theta lift of $\mathcal{D}_{n+1}$ to $O(2n,2)$. If $n>1$, then the restriction of $\pi$ to $SO(2n,2)$ is isomorphic to $A_\mathfrak{q}$. If $n=1$, then $\pi \cong \pi(2,-2)$, where the restriction of $\pi(2,-2)$ to $SO(2,2)^+$ equals $\mathcal{D}_2 \boxtimes \mathcal{D}_{-2} \oplus \mathcal{D}_{-2} \boxtimes \mathcal{D}_2$.
\end{proposition}
\begin{proof} The first claim follows from the computation of the pair $(\mathfrak{q},\lambda=0)$ in \cite[\textsection 6, ($I_4$)]{LiCohomology}. The second claim is \cite[Prop. 4.4.2]{HarrisKudlaGSp2}.
\end{proof}

Recall that underlying the archimedean theta correspondence for the dual pair $(Sp_{2r}(\mathbb{R}),O(2n,2))$ there is a bijective correspondence for $K$-types. This correspondence is explicitly described in \cite[Prop. 4.2.1]{HarrisKudlaGSp2} in terms of highest weights, and shows that the $O(2n)\times O(2)$-type of $\theta(\mathcal{D}_{n+1})$ corresponding to the lowest weight of $\mathcal{D}_{n+1}$ is given by
\begin{equation}
(2,0,\ldots,0)_+ \boxtimes (0)_+.
\end{equation}
(See \cite[I.2]{MoeglinArchimedeanTheta} for the definition of the irreducible representations $(\ldots)_+$.) Moreover, it appears with multiplicity one in $\theta(\mathcal{D}_{n+1})$. This is obvious when $n=1$ by the above proposition, and it holds for any $n>1$ since its restriction to $SO(2n) \times SO(2)$ agrees with the $K$-type $\mu(\mathfrak{q})$ of $A_\mathfrak{q}$ defined in \cite[(2.4)]{VoganZuckerman}, which is a Vogan lowest $K$-type for $A_\mathfrak{q}$ by \cite[Thm 5.3]{VoganZuckerman}.

The representations $A_\mathfrak{q}$ (resp. $A_{\mathfrak{q}^+}$ and $A_{\mathfrak{q}^-}$) also appear in the theta correspondence for the dual pair $(Sp_4(\mathbb{R}),O(2n,2))$ with $n>1$ (resp. with $n=1$); we denote the corresponding $(\mathfrak{sp}_4,U(2))$-module by $\pi_{\infty,n}$. Namely, let $P'=M'N' \subset Sp_4(\mathbb{R})$ be a parabolic subgroup with Levi subgroup $M' \cong SL_2(\mathbb{R}) \times GL_1(\mathbb{R})$ and consider the representation $\mathcal{D}_{n+1} \boxtimes |\det(\cdot)|^{n-1}$ of $M'$; we regard it as a representation of $P'$ by letting $N'$ act trivially. Then $\pi_{\infty,n}$ is an irreducible subquotient of the induced representation $Ind_{P'}^{Sp_4(\mathbb{R})}(\mathcal{D}_{n+1} \boxtimes |\det(\cdot)|^{n-1})$; see \cite[Cor. III.8.(iii)]{MoeglinArchimedeanTheta} for a proof.

The proof of \autoref{thm:Rallis_pairing} requires that a certain $U(2)$-type appears with multiplicity one in $\pi_{\infty,n}$ and a similar property for the theta lift $\theta(\mathbbm{1})$ of the trivial representation of $O(2n+2)$; this is the content of the following lemma.

\begin{lemma} \label{lemma:archimedean_K_type_one} For integers $l \leq l'$, let $\sigma_{l',l}$ be the irreducible representation of $U(2)$ with highest weight $(l',l)$ (see Section \ref{subsection:K_infty_type_varphi}).
\begin{enumerate}
\item The representation $\sigma_{n+1,n-1}$ appears with multiplicity one in $\pi_{\infty,n}$.
\item Consider now the reductive dual pair $(Sp_4(\mathbb{R}),O(2n+2))$. Denote by $\mathbbm{1}$ the trivial representation of $O(2n+2)$ and by $\theta(\mathbbm{1})$ its (small) theta lift. The representation $\sigma_{n+1,n+1}$ appears with multiplicity one in $\theta(\mathbbm{1})$.
\end{enumerate}
\end{lemma}
\begin{proof}
$(1)$ It follows from the description of $K$-types of $A_\mathfrak{q}$ in \cite[Thm 5.3.(c)]{VoganZuckerman} that the $K$-type $(2,0,\ldots,0)_+ \boxtimes (0)_+$ considered above is of minimal Howe degree, and hence its corresponding $U(2)$-type is
\begin{equation*} 
\theta((2,0,\ldots,0)_+ \boxtimes (0)_+) = \sigma_{n+1,n-1},
\end{equation*}
with Howe degree $d=2$. Thus $\sigma_{n+1,n-1}$ appears with multiplicity at least one in $\pi_{\infty,n}$.
To prove the reverse inequality, consider the induced representation $Ind_{P'}^{Sp_4(\mathbb{R})}(\mathcal{D}_{n+1} \boxtimes |\det(\cdot)|^{n-1})$. By the Harish-Chandra decomposition $Sp_4(\mathbb{R})=P' \cdot U(2)$, restriction of functions from $Sp_4(\mathbb{R})$ to $U(2)$ induces an isomorphism
\begin{equation*}
Ind_{P'}^{Sp_4(\mathbb{R})}(\mathcal{D}_{n+1} \boxtimes |\det(\cdot)|^{n-1}) \overset{\cong}{\rightarrow} Ind_{P' \cap U(2)}^{U(2)}(\mathcal{D}_{n+1} \boxtimes |\det(\cdot)|^{n-1}|_{U(2)}) \subset L^2(U(2),\mathcal{D}_{n+1}),
\end{equation*}
where $L^2(U(2),\mathcal{D}_{n+1})$ is the space of square-integrable functions on $U(2)$ valued in $\mathcal{D}_{n+1}$.
Since $\mathcal{D}_{n+1}$ has lowest weight $n+1$, we conclude from the Peter-Weyl theorem that $\sigma_{n+1,n-1}$ appears with multiplicity one in $\pi_{\infty,n}$ in $Ind_{P'}^{Sp_4(\mathbb{R})}(\mathcal{D}_{n+1} \boxtimes |\det(\cdot)|^{n-1}))$. The assertion follows since $\pi_{\infty,n}$ is a subquotient of $Ind_{P'}^{Sp_4(\mathbb{R})}(\mathcal{D}_{n+1} \boxtimes |\det(\cdot)|^{n-1})$.

$(2)$ The multiplicity in the statement is at least $1$ by \cite[Prop. 4.2.1]{HarrisKudlaGSp2} and at most $1$ by \cite[Prop. 2.1]{KudlaRallisDegenerate}.
\end{proof}

\subsection{The $K_\infty$-type of $\tilde{\varphi}_\infty$} \label{subsection:K_infty_type_varphi} Let us briefly recall some of the results in \cite[\textsection 3.10]{GarciaRegularizedLiftsI}. There we constructed a $(1,1)$-form $\Phi(T,\varphi,s)_K$ depending on a totally positive matrix $T \in Sym_2(F)$, a Schwartz function $\varphi \in \mathcal{S}(V(\mathbb{A}_f)^2)$ fixed by $K$ and a complex parameter $s$; this $(1,1)$-form is defined on an open subset $U \subset X_K$ whose complement has measure zero and is integrable on $X_K$, hence defines a current $[\Phi(T,\varphi,s)_K] \in \mathcal{D}^{1,1}(X_K)$. We also defined a function
\begin{equation}
\widetilde{\mathcal{M}}_T(s):N(F) \backslash N(\mathbb{A}) \times A(\mathbb{R})^0 \rightarrow \mathbb{C},
\end{equation}
where $N$ is the unipotent radical of the Siegel parabolic of $Sp_4$ and $A$ denotes the subgroup of diagonal matrices in $Sp_{4,F}$. Given a measurable function $f:Sp_4(\mathbb{A}_F) \rightarrow \mathbb{C}$ that is left invariant under $N(F)$, we defined
\begin{equation} \label{eq:def_theta_lift}
(\widetilde{\mathcal{M}}_T(s),f)^{reg}=\int_{A(\mathbb{R})^0}\int_{N(F) \backslash N(\mathbb{A})} \widetilde{\mathcal{M}}_T(na,s) f(na)dn da
\end{equation}
for certain measures $dn$ and $da$, provided that the integral converges. We also introduced a Schwartz function
\begin{equation} \label{eq:def_tilde_varphi_infty}
\tilde{\varphi}_\infty \in [\mathcal{S}(V(\mathbb{R})^2) \otimes \mathcal{A}^{1,1}(\mathbb{D})]^{H(\mathbb{R})}.
\end{equation}
Given a Schwartz function $\varphi \in \mathcal{S}(V(\mathbb{A}_f)^2)$ fixed by an open compact subgroup $K \subset H(\mathbb{A}_f)$ and $g \in Sp_4(\mathbb{A}_F)$, one can define a theta function
\begin{equation}
\theta(g,\varphi \otimes \tilde{\varphi}_\infty)_K \in \mathcal{A}^{1,1}(X_K)
\end{equation}
that is left invariant under $Sp_4(F)$, and we proved that
\begin{equation}
\tilde{\Phi}(T,\varphi,s)_K=(\widetilde{\mathcal{M}}_T(s),\theta(\cdot,\varphi \otimes \tilde{\varphi}_\infty)_K)^{reg}
\end{equation}
when $Re(s) \gg 0$.

In this section we will describe the behaviour of $\tilde{\varphi}_\infty$ under the action of a certain maximal compact subgroup $K_\infty \subset G(\mathbb{R})$. Recall that this Schwartz function has the form
\begin{equation}
\tilde{\varphi}_\infty(v,w,z)=\tilde{\varphi}^{1,1}(v_1,w_1,z) \otimes \varphi_0^+(v_2,w_2) \otimes \cdots \otimes \varphi_0^+(v_d,w_d).
\end{equation}
Here $\varphi_0^+$ denotes the standard Gaussian on the quadratic vector space $V_i^2 =(V \otimes_{F,\sigma_i} \mathbb{R})^2$, $i=2,\ldots,d$, and
\begin{equation}
\tilde{\varphi}^{1,1} \in [\mathcal{S}(V_1^2) \otimes \mathcal{A}^{1,1}(\mathbb{D})]^{H(\mathbb{R})}
\end{equation}
is of the form $\tilde{\varphi}^{1,1}=\varphi^0 \otimes \varphi_{KM}$, where $\varphi^0$ denotes the Siegel gaussian function and $\varphi_{KM}$ is one of the Schwartz forms introduced by Kudla and Millson. 

 The group $K_\infty$ is defined as follows. The assignment
\begin{equation}
A+iB \in U(2) \mapsto \left( \begin{array}{cc}A & B \\ -B & A \end{array}  \right) \in Sp_4(\mathbb{R}) 
\end{equation}
defines an embedding $U(2) \hookrightarrow Sp_4(\mathbb{R})$ realizing $U(2)$ as a maximal compact subgroup of $Sp_4(\mathbb{R})$. We denote by $K_\infty \subset G(\mathbb{R}) \cong Sp_4(\mathbb{R})^d$ the maximal compact subgroup of $Sp_4(\mathbb{R})^d$ defined by $U(2)^d$. Let $T$ be the maximal torus of $U(2)$ obtained as the image of the embedding $\iota:U(1)^2 \rightarrow U(2)$ defined by
\begin{equation}
(z_1,z_2) \in U(1)^2 \mapsto \left( \begin{array}{cc} z_1 &  \\ & z_2 \end{array}  \right) \in U(2).
\end{equation}
Thus characters of $T$ are parametrized by pairs of integers $(n,m) \in \mathbb{Z}^2$. Given a representation $\sigma$ of $U(2)$, we say that $v \in \sigma$ has weight $(n,m)$ if
\begin{equation}
\iota(z_1,z_2) \cdot v=z_1^n z_2^m \cdot v,
\end{equation}
and if such a $v$ exists we say that $\sigma$ contains the weight $(n,m)$. The irreducible representations of $U(2)$ are parametrized by their highest weights: for a pair of integers $(l,l')$ with $l \leq l'$, there is a unique irreducible representation $\sigma_{(l,l')}$ of $U(2)$ whose set of weights equals $\{(l+k,l'-k)|0 \leq k \leq l'-l \}$; moreover, every irreducible representation of $U(2)$ is of this form. The weights in this set appear with multiplicity one in $\sigma_{(l,l')}$ and hence $dim_\mathbb{C} \sigma_{(l,l')}=l'-l+1$. In particular, $\sigma_{(l,l')}$ contains a vector of weight $(l,l')$ (resp. $(l',l)$); this vector is unique (up to multiplication by scalars) and is called a highest (resp. lowest) weight vector of $\sigma_{(l,l')}$.

\begin{lemma} \phantomsection \label{lemma:tilde_varphi_infty_K_type}
\begin{enumerate}
\item The form $\varphi^0_+$ generates the one-dimensional representation of $U(2)$ with weight $((n+2)/2,(n+2)/2)$.
\item The form $\tilde{\varphi}^{1,1}$ generates an irreducible representation of $U(2)$ of highest weight $((n-2)/2,(n+2)/2)$. Moreover $\tilde{\varphi}^{1,1}$ is a highest weight vector in this representation.
\end{enumerate}
\end{lemma}
\begin{proof}
Part (1) is \cite[(7.19)]{KudlaOrthogonal}. To prove part (2), note first that $\tilde{\varphi}^{1,1}=\varphi^0 \otimes \varphi_{KM}$ has weight $((n-2)/2,(n+2)/2)$ by \cite[(7.7) and Thm. 7.1.ii)]{KudlaOrthogonal}. To study $U(2)$-types, it is more convenient to use a different realization of the Weil representation known as the Fock model. Consider the Harish-Chandra decomposition $\mathfrak{so}(n,2)_\mathbb{C}=Lie(K_{z_0})_\mathbb{C} \oplus \mathfrak{p}_+ \oplus \mathfrak{p}_-$ and denote by $\mathfrak{p}_{\pm}^*$ the dual of $\mathfrak{p}_{\pm}$. We regard $\tilde{\varphi}^{1,1}$ as an element of $\mathcal{S}(V_1^2) \otimes \mathfrak{p}_+^* \otimes \mathfrak{p}_-^*$ via the isomorphism
\begin{equation*}
[\mathcal{S}(V_1^2) \otimes \mathcal{A}^{1,1}(\mathbb{D})]^{H(\mathbb{R})} \cong [\mathcal{S}(V_1^2) \otimes  \mathfrak{p}_+^* \otimes \mathfrak{p}_-^*]^{K_{z_0}}.
\end{equation*}
The form $\tilde{\varphi}^{1,1}$ belongs to the smaller subspace 
\[
S(V_1^2) \otimes  \mathfrak{p}_+^* \otimes \mathfrak{p}_-^* \subset \mathcal{S}(V_1^2) \otimes  \mathfrak{p}_+^* \otimes \mathfrak{p}_-^*,
\]
where, for $r \geq 1$, the space $S(V_1^r)$ consists of all Schwartz forms $\varphi$ of the form $\varphi(v)=P(v) \cdot \varphi^0(v,z_0)$ with $P$ a polynomial on $V_1^r$. This space $S(V_1^r)$ is preserved by the action of the complexified Lie algebra $\mathfrak{sp}(2r)_\mathbb{C} \times \mathfrak{so}(n,2)_\mathbb{C}$ and admits an intertwining map
\[
\iota:S(V_1^{r}) \rightarrow \mathcal{P}(\mathbb{C}^{r(n+2)})
\]
such that $\iota(\varphi^0(\cdot,z_0))=1$, where $\mathcal{P}(\mathbb{C}^{r(n+2)})$ is the space of polynomials on $\mathbb{C}^{r(n+2)}$ endowed with the action $\mathfrak{sp}(2r)_\mathbb{C} \times \mathfrak{so}(n,2)_\mathbb{C}$ described in \cite[Thm. 7.1]{KudlaMillson3}. The image of the form $\varphi_{KM}(\cdot,z_0) \in S(V_1)$ under $\iota$ is computed explicitly in \cite[(4.2)]{BruinierFunke} and is shown to be a ($\mathfrak{p}_+ \otimes \mathfrak{p}_-$-valued) polynomial of degree $2$ on $\mathbb{C}^{n+2}$. It follows that the form $\tilde{\varphi}^{1,1}$ corresponds under $\iota$ to a $\mathfrak{p}_+ \otimes \mathfrak{p}_-$-valued polynomial of degree $2$ on $\mathbb{C}^{2(n+2)}$.

The action of $\mathfrak{u}(2)$ on $\mathcal{P}(\mathbb{C}^{2(n+2)})$ preserves degrees and integrates to an action of $U(2)$. Hence, for an irreducible representation $\sigma$ of $U(2)$, we can define $deg(\sigma)$ to be the minimal degree $d$ such that $\sigma$ appears in $\mathcal{P}_d(\mathbb{C}^{2(n+2)})$, the subspace of polynomials of degree $d$. Note that the irreducible representations $\sigma$ containing the weight $((n-2)/2,(n+2)/2)$ are of the form $\sigma_{l,l'}$, where $(l,l')=((n-2)/2-k,(n+2)/2+k)$ for some $k \geq 0$. By \cite[Prop. 4.2.1]{HarrisKudlaGSp2}, the degree of such $\sigma$ equals
\[
deg(\sigma_{(l,l')})=2+2k
\]
and hence the representation generated by $\iota(\tilde{\varphi}^{1,1})$ equals $\sigma_{((n-2)/2,(n+2)/2)}$ as was to be shown.
\end{proof}

The lemma implies that, under the action of $K_\infty \cong U(2)^d$, the form $\tilde{\varphi}_\infty$ generates an irreducible representation $\sigma_\infty$ given by
\begin{equation} \label{eq:def_sigma_infty}
\sigma_\infty = \sigma_{((n-2)/2,(n+2)/2)} \boxtimes \sigma_{((n+2)/2,(n+2)/2)} \boxtimes \cdots \boxtimes \sigma_{((n+2)/2,(n+2)/2)},
\end{equation}
and that $\tilde{\varphi}_\infty$ corresponds to a highest weight vector in $\sigma_\infty$.

\subsection{Theta lifts valued in differential forms} \label{subsection:theta_lifts_diff_forms} In this section we define some differential forms of degree $(n-1,n-1)$ on $X_K$ obtained as theta lifts of cusp forms on $Sp_4(\mathbb{A}_F)$. To do so, we first introduce a certain Schwartz form $\varphi_\infty$ in $\mathcal{S}(V(\mathbb{R})^2)$ valued in the space $\mathcal{A}^{n-1,n-1}(\mathbb{D})$ of differential forms of degree $(n-1,n-1)$ on $\mathbb{D}$. Then we define theta functions and theta lifts associated to this Schwartz form.

We start with the definition of $\varphi_\infty$. Let $k_D(z_1,z_2)$ be the Bergmann kernel of $\mathbb{D}$ and denote by $\Omega$ the K\"ahler form
\begin{equation}
\Omega(z)=\partial \overline{\partial}\log k_D(z,z) \in \mathcal{A}^{1,1}(\mathbb{D}).
\end{equation}
Then $(2\pi i)^{-1}\Omega$ is invariant under $H(\mathbb{R})$ and descends to a $(1,1)$-form on $X_K$ that represents the first Chern class of the canonical bundle of $X_K$. In particular, the cup product with $(2\pi i)^{-1}\Omega$ induces the Lefschetz operator 
\begin{equation}
H^*(X_K,\mathbb{C}) \rightarrow H^{*+2}(X_K,\mathbb{C})
\end{equation}
in cohomology. We now define a Schwartz form
\begin{equation} \label{eq:def_varphi_infty_n-1_1}
\varphi_\infty \in [\mathcal{S}(V(\mathbb{R})^2) \otimes \mathcal{A}^{n-1,n-1}(\mathbb{D})]^{H(\mathbb{R})}
\end{equation}
by
\begin{equation} \label{eq:def_varphi_infty_n-1_2}
\varphi_\infty=\tilde{\varphi}_{\infty} \wedge \Omega^{n-2},
\end{equation}
with $\tilde{\varphi}_\infty$ is defined in \eqref{eq:def_tilde_varphi_infty}.

Denote by $\omega$ the adelic Weil representation of $Sp_4(\mathbb{A}) \times O(V)(\mathbb{A})$ on $\mathcal{S}(V(\mathbb{A})^2)$. For $\varphi \in \mathcal{S}(V(\mathbb{A}_f)^2)$, $g \in Sp_4(\mathbb{A}_F)$ and $h_f \in O(V)(\mathbb{A}_f)$, define
\begin{equation}
\theta(g,h_f;\varphi \otimes \varphi_\infty)=\sum_{x \in V(F)^2} \omega(g_f,h_f)\varphi(x) \cdot \omega(g_\infty)\varphi_\infty(x,z).
\end{equation}
The sum converges absolutely and defines a differential form 
\begin{equation}
\theta(g;\varphi \otimes \varphi_\infty)_K \in \mathcal{A}^{n-1,n-1}(X_K),
\end{equation}
for any compact open $K \subset O(V)(\mathbb{A}_f)$ fixing $\varphi$. Moreover, as a function of $g \in Sp_4(\mathbb{A}_F)$ it is left invariant under $Sp_4(F)$ and of moderate growth on $Sp_4(F)\backslash Sp_4(\mathbb{A}_F)$. For a cusp form $f$ on $Sp_4(F) \backslash Sp_4(\mathbb{A}_F)$, define
\begin{equation} \label{eq:def_theta_lift_n-1_K}
\theta(f,\varphi \otimes \varphi_\infty)_K=\int_{Sp_4(F) \backslash Sp_4(\mathbb{A}_F)}\overline{f(g)} \theta(g;\varphi \otimes \varphi_\infty)_K dg. 
\end{equation}
Since $f$ is rapidly decreasing on $Sp_4(F) \backslash Sp_4(\mathbb{A}_F)$, the integral converges absolutely and defines a form in $\mathcal{A}^{n-1,n-1}(X_K)$. Note that if $K' \subset K$, then there is a natural map $\iota_{K',K}: X_{K'} \rightarrow X_K$ and we have
\begin{equation}
\theta(g;\varphi \otimes \varphi_\infty)_{K'} = \iota^*_{K',K}(\theta(g;\varphi \otimes \varphi_\infty)_K),
\end{equation}
so that we can define
\begin{equation} \label{eq:def_theta_lift_n-1}
\theta(f,\varphi \otimes \varphi_\infty)=(\theta(f,\varphi \otimes \varphi_\infty)_K)_K \in \mathcal{A}^{n-1,n-1}(X)=  \varinjlim_{K} \mathcal{A}^{n-1,n-1}(X_K).
\end{equation}

Suppose that $f$ generates an irreducible, cuspidal automorphic representation $\pi=\otimes'_v \pi_v$ of $Sp_4(\mathbb{A}_F)$. By choosing the archimedean components $\pi_v$ appropriately, we can ensure that $\theta(f,\varphi \otimes \varphi_\infty)$, if non-vanishing, defines a non-zero cohomology class in $H^{n-1,n-1}(X)$. Namely, assume that for $v$ archimedean, we have
\begin{equation} \label{eq:pi_arch_choice}
\pi_v= \left\{ \begin{array}{ccc} \pi_{\infty,n}, & \quad \text{for} \quad v=\sigma_1, & \\
\theta(\mathbbm{1}), & \quad \text{for} \quad v=\sigma_i, & \quad i=2,\ldots,d.
\end{array}  \right.
\end{equation}
Here $\pi_{\infty,n}$ is defined in \autoref{subsection:reps_and_cohomology} and $\theta(\mathbbm{1})$ is the theta lift to $Sp_4(\mathbb{R})$ of the trivial representation $\mathbbm{1}$ of $O(n+2)$. The results of Section \ref*{subsection:reps_and_cohomology} show that, for $n>1$ (resp. $n=1$), the restriction of the theta lift $\theta(\pi_{\sigma_1})$ to $SO(n,2)$ equals $A_\mathfrak{q}$ (resp. $A_{\mathfrak{q}^+} \oplus A_{\mathfrak{q}^-}$). It follows that the eigenvalue of the Casimir element on the global theta lift $\theta(\pi)$ equals $0$; in particular, the form $\theta(f,\varphi \otimes \varphi_\infty)$ is harmonic. Since the Shimura variety $X_K$ is compact by assumption, an application of Hodge theory shows that the form $\theta(f,\varphi \otimes \varphi_\infty)$ is closed and that, if non-vanishing, it defines a non-trivial cohomology class in $H^{n-1,n-1}(X):=\varinjlim_{K} H^{n-1,n-1}(X_K)$.

We recall some results of \cite[\textsection 10]{KudlaOrthogonal} on the integration of differential forms on $X_K$. Let $\mathfrak{so}(n,2)$ the Lie algebra of $SO(n,2)^+$ and $\mathfrak{k}$ be the complexification of the Lie algebra of its maximal compact subgroup $K_{n,2} = SO(n) \times SO(2)$. Let $\mathfrak{p}_+$, $\mathfrak{p}_-$ be the summands in the Harish-Chandra decomposition $\mathfrak{so}(n,2)_\mathbb{C}=\mathfrak{k} \oplus \mathfrak{p}_+ \oplus \mathfrak{p}_-$ and write $\bigwedge^{a,b}\mathfrak{p}^*=\bigwedge^a \mathfrak{p}_+^* \otimes \bigwedge^b \mathfrak{p}_-^*$. Note that
\begin{equation}
\mathcal{A}^{a,b}(\mathbb{D}) \cong [\bigwedge^{a,b}\mathfrak{p}^* \otimes \mathcal{C}^\infty(SO(n,2))]^{K_{n,2}}.
\end{equation} 
Denote by $Z \cong Res_{F/\mathbb{Q}}\mathbb{G}_m$ the kernel of the natural map $H \rightarrow Res_{F/\mathbb{Q}}SO(V)$. A choice of a non-zero element $\mathbbm{1} \in \bigwedge^{n,n}\mathfrak{p}$ yields an isomorphism 
\begin{equation} \label{eq:isom_volume_forms_aut_forms}
\begin{split}
\mathcal{A}^{n,n}(X_K) & \cong [\mathcal{A}^{n,n}(\mathbb{D}) \otimes \mathcal{C}^\infty(H(\mathbb{A}_f))]^{H(\mathbb{Q}) \times K} \\
& \cong [\bigwedge^{n,n} \mathfrak{p}^* \otimes \mathcal{C}^\infty(SO(n,2)) \otimes \mathcal{C}^\infty(H(\mathbb{A}))]^{H(\mathbb{Q}) \times K \times K_{n,2}} \\
& \cong [\bigwedge^{n,n} \mathfrak{p}^* \otimes \mathcal{C}^\infty(H(\mathbb{Q})Z(\mathbb{R}) \backslash H(\mathbb{A}))]^{K \times K_{n,2} \times SO(n+2)^{d-1}} \\
& \cong [\mathcal{C}^\infty(H(\mathbb{Q})Z(\mathbb{R}) \backslash H(\mathbb{A}))]^{K \times K_{n,2} \times SO(n+2)^{d-1}}.
\end{split}
\end{equation}
This isomorphism allows to identify the forms $\eta \in \mathcal{A}^{n,n}(X_K)$ of top degree on $X_K$ with automorphic forms $\tilde{\eta}$ on $H$. Under this identification, we have
\begin{equation} \label{eq:def_c}
\int_{X_K} \eta = c' \cdot Vol(K/K \cap Z(\mathbb{Q}),dh)^{-1} \cdot \int_{H(\mathbb{Q})Z(\mathbb{R}) \backslash H(\mathbb{A})} \tilde{\eta}(h) dh,
\end{equation}
where $dh$ denotes the Tamagawa measure on $H$ and $c'>0$ is a constant independent of $K$; this is \cite[Lemma 10.2]{KudlaOrthogonal}.

\subsection{Pairing currents and forms} \label{subsection:pairing_currents_forms} Let $\pi=\otimes'_v\pi_v \subset \mathcal{A}_0(Sp_{4,F})$ be an irreducible cuspidal automorphic representation of $Sp_4(\mathbb{A}_F)$ and $f \in \pi$ be a cusp form. Choose Schwartz forms $\varphi, \tilde{\varphi} \in \mathcal{S}(V(\mathbb{A}^\infty)^2)$ and let $T \in Sym_2(F)$ be totally positive definite. We have now introduced currents
\begin{equation*}
[\tilde{\Phi}(T,\tilde{\varphi},s')] \in \mathcal{D}^{1,1}(X)
\end{equation*}
and differential forms
\begin{equation*}
\theta(f,\varphi \otimes \varphi_\infty) \in \mathcal{A}^{n-1,n-1}(X).
\end{equation*}
Define
\begin{equation} \label{eq:def_pairing_limit}
([\tilde{\Phi}(T,\tilde{\varphi},s')],\overline{\theta(f,\varphi \otimes \varphi_\infty)})=Vol(K/K\cap Z(\mathbb{Q}))\cdot \int_{X_K} \tilde{\Phi}(T,\tilde{\varphi},s')_K \wedge \overline{\theta(f,\varphi \otimes \varphi_\infty)_K},
\end{equation}
where $K$ is any compact open subgroup of $H(\mathbb{A}_f)$ such that $\tilde{\varphi}$ and $\varphi$ are fixed by $K$. It is easy to see that this definition is independent of the choice of $K$. 

To relate this pairing to a special value of an $L$-function, we will need to restrict our attention to Schwartz forms $\tilde{\varphi}$ belonging to a certain $K^\infty$-type of the Weil representation. More precisely, fix an open compact subgroup $K^\infty=\Pi_{v \nmid \infty} K_v$ of $Sp_4(\mathbb{A}_F^\infty)$ and define $K=K^\infty \times U(2)^d$. Let $dk=\Pi_v dk_v$ be a Haar measure on $K^\infty$ defined by local measures such that $Vol(K_v,dk_v)=1$ for every place $v$. Fix an irreducible representation $\sigma^\infty=\otimes'_{v \nmid \infty} \sigma_v$ of $K^\infty$. Thus $\sigma^\infty$ is finite dimensional and for $v$ outside some finite set of places, the local component $\sigma_v$ is the trivial representation. We write 
\begin{equation}\label{eq:def_sigma}
\sigma=\sigma^\infty \otimes \sigma_\infty,
\end{equation}
where $\sigma_\infty$ is the representation of $U(2)^d$ given by \eqref{eq:def_sigma_infty}, and we fix a highest weight vector $w_\infty=\otimes_{v | \infty}w_v \in \sigma_\infty$ and a lowest weight vector $w_\infty^\vee=\otimes_{v | \infty}w_v^\vee \in \sigma_\infty^\vee$ with $(w_\infty,w^\vee_\infty)=1$.

\begin{definition} \label{def:varphi_wveew}
For vectors $w \in \sigma^\infty$ and $w^\vee \in (\sigma^\infty)^\vee$ and a Schwartz form $\tilde{\varphi} \in \mathcal{S}(V(\mathbb{A}^\infty)^2)$, define a Schwartz form $\tilde{\varphi}_{w^\vee,w}$ by
\begin{equation} \label{eq:def_varphi_wveew}
\tilde{\varphi}_{w^\vee,w}(x,y)=\int_{K^\infty} (\sigma^\vee(k)w^\vee,w) \cdot \omega(k)\tilde{\varphi}(x,y)dk \in \mathcal{S}(V(\mathbb{A}^\infty)^2).
\end{equation}
\end{definition}
Thus, for any $w^\vee$ and $w$, the form $\tilde{\varphi}_{w^\vee,w}$ belongs to the $\sigma$-isotypic part of the Schwartz space $\mathcal{S}(V(\mathbb{A}^\infty)^2)$. If $w=\otimes'_v w_v$, $w^\vee=\otimes'_v w_v^\vee$ and $\tilde{\varphi}=\otimes_v' \tilde{\varphi}_v$ are pure tensors, then $\tilde{\varphi}_{w^\vee,w}=\otimes'_v (\tilde{\varphi}_v)_{w^\vee_v,w_v}$ where we define
\begin{equation} \label{eq:def_varphi_wveew_local}
(\tilde{\varphi}_v)_{w_v^\vee,w_v}=\int_{K_v} (\sigma_v^\vee(k)w_v^\vee,w_v) \cdot \omega(k)\tilde{\varphi}_vdk_v \in \mathcal{S}(V(F_v)^2).
\end{equation}

To state our main theorem, we need to introduce some additional notation. Namely, let $H'=Res_{F / \mathbb{Q}}O(V)$ and $H^0= Res_{F / \mathbb{Q}}SO(V)$. Then $H^0$ is the connected component of the identity of $H'$ and we have
\begin{equation}
H'=H^0 \rtimes \langle \tau \rangle,
\end{equation}
where $\langle \tau \rangle \cong \mathbb{Z}/2\mathbb{Z}$ and one can take $\tau$ to be the reflection along any hyperplane $W$ of $V$. We fix such a $W$ and assume $W$ to be defined over $F$ and $z_0 \perp W$; $\tau$ then denotes the reflection along $W$ and for any set of places $S$ we write
\begin{equation}
C(\mathbb{A}_S)=\prod_{v \in S} \langle \tau \rangle \cong H'(\mathbb{A}_S)/H^0(\mathbb{A}_S).
\end{equation}
Denote by 
\begin{equation} \label{eq:def_alpha_v_arch}
\alpha_\infty \in \mathcal{S}(V(\mathbb{R})^4)
\end{equation}
the Schwartz form obtained by evaluating the form
\begin{equation}
\tilde{\varphi}_\infty \wedge \overline{\varphi_\infty} \in [\mathcal{S}(V(\mathbb{R})^4) \otimes \mathcal{A}^{n,n}(\mathbb{D})]^{H(\mathbb{R})} \cong  [\mathcal{S}(V(\mathbb{R})^4) \otimes \bigwedge^{2n}\mathfrak{p}^*]^{K_{z_0}}
\end{equation}
on a fixed element $\mathbbm{1} \in \bigwedge^{2n}\mathfrak{p}$; then $\alpha_\infty=\otimes_{v|\infty}\alpha_v$. The following property of $\alpha_\infty$ will be used in the proof of \autoref{thm:Rallis_pairing}.

\begin{lemma} \label{lemma:C_R_invariance}
The form $\alpha_\infty \in \mathcal{S}(V(\mathbb{R})^4)$ is invariant under $C(\mathbb{R})=\Pi_{v | \infty} \langle \tau \rangle$.
\end{lemma}
\begin{proof}
The forms $\varphi^0(\cdot,z_0) \in \mathcal{S}(V_1)$ and $\varphi^0_+ \in \mathcal{S}(V_{\sigma_i}^2)$ for $i=2,\ldots,d$ defined in 
\cite[\textsection 3.8.1]{GarciaRegularizedLiftsI} are clearly invariant under $\tau$. Since $\overline{\varphi_{KM}}=-\varphi_{KM}$, it suffices to show that
\begin{equation*}
\omega(v,w)=(\varphi_{KM}(v,z_0) \wedge \varphi_{KM}(w,z_0) \wedge \Omega(z_0))(\mathbbm{1}) \in \mathcal{S}(V_1^2)
\end{equation*}
is invariant under $\tau$. Consider $V_1^2$ as embedded in $V_1^n$ by the map sending $(v,w)$ to $(v,w,0,\ldots,0)$. By \cite[Thm 7.1.(iv) and (v)]{KudlaOrthogonal}, the Schwartz function $\omega$ is a multiple of the restriction to $V_1^2$ of $\varphi^{(n)}(\mathbbm{1}) \in \mathcal{S}(V_1^n)$. The proof of \cite[Lemma 10.4] {KudlaOrthogonal} shows that $\varphi^{(n)}(\mathbbm{1})$, and hence $\omega$, is invariant under $\tau$.
\end{proof}

From now on we assume that the following hypotheses hold; note that, by the results of \autoref{subsection:reps_and_cohomology}, hypotheses $(1)$ and $(2)$ imply that the theta lift $\theta(\pi)$ contributes non-trivial cohomology classes in $H^{n-1,n-1}(X)$. Note also that hypotheses $(2)-(5)$ are local.

\begin{hypotheses} \phantomsection \label{hyp:main_thm_pairing}
\begin{enumerate}
\item The global theta lift $\theta(\pi) \subset \mathcal{A}(O(V))$ is non-vanishing;
\item At archimedean places $v$, the representation $\pi_v$ is given by \eqref{eq:pi_arch_choice};
\item $\dim_\mathbb{C} Hom_{K^\infty}(\sigma^\infty,\pi^\infty)=1$;
\item The forms $f=\otimes_v f_v$, $\varphi=\otimes'_{v \nmid \infty} \varphi_v$ and $\tilde{\varphi}=\otimes'_{v \nmid \infty} \tilde{\varphi_v}$ are pure tensors;
\item The Schwartz form $\otimes_{v \nmid \infty} (\tilde{\varphi}_v \otimes \varphi_v)$ is $C(\mathbb{A}^\infty)$-invariant.
\end{enumerate} 
\end{hypotheses}

For a non-archimedean place $v$ and fixed $w_v \in \sigma_v$, $w_v^\vee \in \sigma_v^\vee$, write
\begin{equation} \label{eq:def_alpha_v_nonarch}
\alpha_v=\tilde{\varphi}_v \otimes \overline{\varphi_v} \in \mathcal{S}(V(F_v)^4).
\end{equation}
We denote by $\Phi_v \in Ind_{P_8(F_v)}^{Sp_8(F_v)}(\chi_V,s)$ the section determined by $\alpha_v$, defined for $g \in Sp_8(F_v)$ by
\begin{equation} \label{eq:def_Phi_v_thm_Rallis}
\Phi_v(g,s)=\omega(g)(\alpha_v)(0) \cdot |a(g)|^{s-\frac{1-m}{2}}.
\end{equation}
and by $\Phi_{w^\vee_v,w_v}$ the section of $Ind_{P_8(F_v)}^{Sp_8(F_v)}(\chi_V,s)$ defined by \eqref{eq:def_Phi_v_wveew}.
For archimedean places $v$, we define $\Phi_v$ by equation \eqref{eq:def_Phi_v_thm_Rallis}, with $\alpha_v \in \mathcal{S}(V(F_v)^4)$ the Schwartz functions determined by the condition $\alpha_\infty= \otimes_{v | \infty} \alpha_v$, with $\alpha_\infty$ as in \eqref{eq:def_alpha_v_arch}.


The results of \autoref{subsection:K_infty_type_varphi} show that the Schwartz form $\tilde{\varphi}_\infty$ gives a highest weight vector in $\sigma_\infty$, and hence we simply define $\Phi_{w_v^\vee,w_v}=\Phi_v$ for $v$ archimedean. (Since integrating $\Phi_\infty$ against the matrix coefficient $(\sigma_\infty^\vee w_\infty^\vee,w_\infty)$ as in \eqref{eq:def_Phi_v_wveew} gives  $1/3 \cdot \Phi_\infty$ by Schur orthogonality.)

We are now ready to state the main result of this section. Its proof will be given in \autoref{subsection:proof_thm_Rallis_pairing}.

\begin{theorem} \label{thm:Rallis_pairing} Assume that the hypotheses in \ref{hyp:main_thm_pairing} hold and choose a $K$-invariant embedding $\iota_\pi=\otimes_v \iota_{\pi,v}: \sigma \hookrightarrow \pi$. Let $w=\otimes_{v \nmid \infty} w_v \in \sigma^\infty$, $w^\vee=\otimes_{v \nmid \infty} w_v^\vee \in (\sigma^\infty)^\vee$ and $f^\vee=\otimes_v f_v^\vee \in \pi^\vee$ such that $(\iota_{\pi,v}(w_v),f_v^\vee)=1$ for all $v$. Then there is a positive constant $C$, depending only on the quadratic space $V$, such that
\begin{equation*}
\begin{split}
([\tilde{\Phi}(T,\tilde{\varphi}_{w^\vee,w},s')],\overline{\theta(f,\varphi \otimes \varphi_\infty)}) & =C \cdot (\widetilde{\mathcal{M}}_T(s'),f_{\iota_\pi(w \otimes w_\infty)})^{reg} \\
& \quad \cdot \left. \frac{\Lambda(\pi^\vee,\chi_V,s+1/2)}{d(\chi_V,s)} \right|_{s=s_0} \cdot \prod_v Z^0_v(f^\vee_v,f_v,\delta_*\Phi_{w_v^\vee,w_v},\chi_{V,v},s_0),
\end{split}
\end{equation*}
where $s_0=(n-3)/2$ and for all but finitely many $v$ we have $Z^0_v(f^\vee_v,f_v,\delta_*\Phi_{w_v^\vee,w_v},\chi_{V,v},s_0)=1$.
\end{theorem}
%
\begin{remarks} \phantomsection \label{remark:on_thm_Rallis_pairing}
\begin{enumerate} 
\item By \autoref{lemma:archimedean_K_type_one}, the representation $\sigma_\infty$ appears with multiplicity one in $\pi_\infty$. Thus hypothesis \ref{hyp:main_thm_pairing}.$(2)$ implies that $\dim_\mathbb{C}Hom_K(\sigma,\pi)=1$. It then follows from \autoref{lemma:int_independence} and \eqref{eq:Z_functional_invariance} that the right hand side is independent of the choice of $f^\vee$ such that $(\iota_\pi(w \otimes w_\infty),f^\vee)=1$ and of the choice of embedding $\iota_\pi$.
\item By \autoref{prop:completed_L_fcn_regular}, the meromorphic function $d(\chi_V,s)^{-1}\Lambda(\pi^\vee,\chi_V,s+1/2)$ is regular and non-vanishing at $s_0$.
\item Hypothesis \ref{hyp:main_thm_pairing}.$(3)$ is necessary to obtain a Euler product ranging over all places $v$ of $F$ in the Theorem. Without assuming the existence of a pair $(\sigma,K)$ appearing with multiplicity one in $\pi$, one only obtains a partial Euler product ranging over unramified places $v$ of $F$. This is essentially the same phenomenon as in \cite{PSRNew} and is related to the fact that the Rankin-Selberg integral unfolds to a non-unique model that does not enjoy a multiplicity one property. Piatetskii-Shapiro and Rallis use the uniqueness of spherical vectors at unramified places to show that their Rankin-Selberg integral factors as a product involving an incomplete $L$-function, and Hypothesis $(3)$ allows to generalize their method to every place $v$ in $F$ and hence obtain the complete $L$-function.

\item In \autoref{section:Example_Products_Shimura_curves} we will consider an example where all the hypotheses are satisfied.
\end{enumerate}
\end{remarks}

Define
\begin{equation} \label{eq:def_Beil_period}
I(T,\tilde{\varphi}_{w^\vee,w};f,\varphi)=CT_{s'=s_0'}(\widetilde{\mathcal{M}}_T(s'),f_{\iota_\pi(w \otimes w_\infty)})^{reg} \cdot \prod_v Z^0(f^\vee_v,f_v,\delta_*\Phi_{w_v^\vee,w_v},\chi_{V,v},s_0).
\end{equation}
Recall that we have shown in \autoref{subsection:reps_and_cohomology} that the hypotheses in \ref{hyp:main_thm_pairing} imply that $\theta(f,\varphi \otimes \varphi_\infty)$ is closed. Combining \autoref{thm:Rallis_pairing} and 
\cite[(3.83)]{GarciaRegularizedLiftsI}, we obtain the following Corollary.
\begin{corollary} \label{cor:thm_Rallis_pairing}
Assume that the hypotheses in the theorem hold and that $\tilde{\varphi}_{w^\vee,w}^\iota=\tilde{\varphi}_{w^\vee,w}$. Then
\begin{equation*}
\begin{split}
([\Phi(T,\tilde{\varphi}_{w^\vee,w})]-[\Phi(T^\iota,\tilde{\varphi}_{w^\vee,w}^\iota)],\overline{\theta(f,\varphi\otimes \varphi_\infty)})&=C \cdot \left. \frac{\Lambda(\pi^\vee,\chi_V,s+1/2)}{d(\chi_V,s)}  \right|_{s=s_0} \\
& \quad \cdot (I(T,\tilde{\varphi}_{w^\vee,w};f,\varphi)-I(T^\iota,\tilde{\varphi}_{w^\vee,w};f,\varphi)).
\end{split}
\end{equation*}
\end{corollary}

\subsection{Proof of \autoref{thm:Rallis_pairing}} \label{subsection:proof_thm_Rallis_pairing} Write $G=Sp_{4,F}$. Let $K' \subset H(\mathbb{A}_f)$ be an open compact subgroup fixing $\varphi$ and $\tilde{\varphi}_{w^\vee,w}$ and write $v(K')=Vol(K'/K' \cap Z(\mathbb{Q}))$. By 
\cite[Prop. 3.27]{GarciaRegularizedLiftsI} we have, for $Re(s') \gg 0$:
\begin{equation*}
\begin{split}
([\tilde{\Phi}(T,\tilde{\varphi}_{w^\vee,w},s')], & \overline{\theta(f,\varphi \otimes \varphi_\infty)}) \\
& = v(K') \cdot \int_{X_{K'}} \tilde{\Phi}(T,\tilde{\varphi}_{w^\vee,w},s')_{K'} \wedge \overline{\theta(f,\varphi \otimes \varphi_\infty)_{K'}} \\
& = v(K') \cdot \int_{A(\mathbb{R})^0}\int_{N(F)\backslash N(\mathbb{A})} \widetilde{\mathcal{M}}_T(na,s') \\
& \qquad \qquad \qquad \qquad \qquad \cdot \int_{X_{K'}} \theta(na;\tilde{\varphi}_{w^\vee,w} \otimes \tilde{\varphi}_\infty)_{K'} \wedge \overline{\theta(f,\varphi \otimes \varphi_\infty)_{K'}} \ dn da.
\end{split}
\end{equation*}
The integral over $X_{K'}$ in the last line equals
\begin{equation*}
\int_{G(F) \backslash G(\mathbb{A}_F)}f(g) \int_{X_{K'}} \theta(na;\tilde{\varphi}_{w^\vee,w} \otimes \tilde{\varphi}_\infty)_{K'} \wedge \overline{\theta(g;\varphi \otimes \varphi_\infty)_{K'}} \ dg.
\end{equation*}
The inner integral in this expression can be rewritten as an integral over the adelic points of $H$ using \eqref{eq:def_c}. 
Namely, let $\iota:Sp_4 \times Sp_4 \rightarrow Sp_8$ be the embedding given by \eqref{eq:def_iota} and denote by $\tilde{\omega}=\tilde{\omega}_\psi$ the Weil representation of $Sp_8(\mathbb{A})$ on $\mathcal{S}(V(\mathbb{A}^4))$ associated with $\psi$. Then
\[
\tilde{\omega}_\psi \circ \iota = \omega_\psi \boxtimes \omega^*_\psi,
\]
where $\omega^*_\psi$ denotes the contragredient of $\omega_\psi$ and is isomorphic to $\omega_{\overline{\psi}}$. Consider the Schwartz form $\alpha_{w^\vee,w}=\alpha_{w^\vee,w}^\infty \otimes \alpha_\infty \in \mathcal{S}(V(\mathbb{A})^4)$, with $\alpha_\infty$ defined by \eqref{eq:def_alpha_v_arch} and 
$$
\alpha_{w^\vee,w}^\infty=\tilde{\varphi}_{w^\vee,w} \otimes \overline{\varphi} \in \mathcal{S}(V(\mathbb{A}^\infty)^4).
$$
For $g \in Sp_8(\mathbb{A}_F)$ and $h \in O(V)(\mathbb{A})$, define a theta function $\theta(g,h;\alpha_{w^\vee,w})$ by
\begin{equation*}
\theta(g,h;\alpha)=\sum_{v \in V(F)^4}\tilde{\omega}_\psi(g,h)\alpha_{w^\vee,w}(v).
\end{equation*}
Then the volume form
\begin{equation*}
\theta(na;\tilde{\varphi}_{w^\vee,w} \otimes \tilde{\varphi}_\infty)_{K'} \wedge \overline{\theta(g;\varphi \otimes \varphi_\infty)_{K'}} \in \mathcal{A}^{n,n}(X_{K'})
\end{equation*}
corresponds to $\theta(\iota(na,g),\cdot;\alpha_{w^\vee,w})$ under the isomorphism \eqref{eq:isom_volume_forms_aut_forms}, so that \eqref{eq:def_c} implies
\begin{equation} \label{eq:from_intX_K_to_int_O(V)_1}
\begin{split}
\int_{X_{K'}} \theta(na;\tilde{\varphi}_{w^\vee,w} \otimes \tilde{\varphi}_\infty)_{K'} \wedge \overline{\theta(g;\varphi \otimes \varphi_\infty)_{K'}}&= c' \cdot v(K')^{-1} \\
&\quad \cdot \int_{H(\mathbb{Q})Z(\mathbb{R}) \backslash H(\mathbb{A})}\theta(\iota(na,g),h;\alpha_{w^\vee,w})dh,
\end{split}
\end{equation}
where $dh$ denotes the Tamagawa measure on $H$ and $c'>0$ is a constant independent of $K$. Note that the integrand is in fact invariant under $Z(\mathbb{A})$, so that writing $dh^0$ (resp. $dz$) for the Tamagawa measure on $H^0=Z\backslash H$ (resp. $Z$), we have
\begin{equation*} \label{eq:from_intX_K_to_int_O(V)_2}
\eqref{eq:from_intX_K_to_int_O(V)_1}=Vol(Z(\mathbb{Q})Z(\mathbb{R}) \backslash Z(\mathbb{A}),dz) \cdot \int_{H^0(\mathbb{Q}) \backslash H^0(\mathbb{A})} \theta(\iota(na,g),h^0;\alpha_{w^\vee,w})dh^0.
\end{equation*}
Since by hypothesis \ref{hyp:main_thm_pairing}.$(4)$ and \autoref{lemma:C_R_invariance} the Schwartz function $\alpha_{w^\vee,w}$ is invariant under $C(\mathbb{A})$, we can replace the integral over $H^0(\mathbb{Q}) \backslash H^0(\mathbb{A})$ in \eqref{eq:from_intX_K_to_int_O(V)_2} with an integral over $H'(\mathbb{Q}) \backslash H'(\mathbb{A})$. Note that if we denote by $dh'$ the Haar measure on $H'(\mathbb{A})$ giving $H'(\mathbb{Q}) \backslash H'(\mathbb{A})$ unit volume, then $Vol(C(\mathbb{Q}) \backslash C(\mathbb{A}),dh'/dh^0)=1/2$, since the Tamagawa measure of $H^0$ equals $2$. It follows that
\begin{equation*}
\begin{split}
\int_{H'(\mathbb{Q}) \backslash H'(\mathbb{A})} \theta(\iota(na,g),h';\alpha_{w^\vee,w})dh' &=\int_{C(\mathbb{Q}) \backslash C(\mathbb{A})}\int_{H^0(\mathbb{Q}) \backslash H^0(\mathbb{A})} \theta(\iota(na,g),h^0 c;\alpha_{w^\vee,w})dh^0 dc \\
&= \frac{1}{2} \cdot \int_{H^0(\mathbb{Q}) \backslash H^0(\mathbb{A})} \theta(\iota(na,g),h^0;\alpha_{w^\vee,w})dh^0
\end{split}
\end{equation*}
and hence
\begin{equation*}\begin{split}
\eqref{eq:from_intX_K_to_int_O(V)_1}= c \cdot v(K')^{-1} \cdot \int_{H'(\mathbb{Q}) \backslash H'(\mathbb{A})} \theta(\iota(na,g),h';\alpha_{w^\vee,w})dh',
\end{split}
\end{equation*}
where 
\begin{equation*}
c=2 \cdot c' \cdot Vol(Z(\mathbb{Q})Z(\mathbb{R}) \backslash Z(\mathbb{A}),dz).
\end{equation*}
The integral on the right hand side can be evaluated using the Siegel-Weil formula. Namely, note that the section of $Ind_{P(\mathbb{A})}^{Sp_8(\mathbb{A})}(\chi_V |\cdot|^s)$ attached to $\alpha_{w^\vee,w}$ is given by
\[
\Phi_{w^\vee,w}= \prod_v \Phi_{w_v^\vee,w_v}.
\]
Let $E(g,\Phi_{w^\vee,w},s)$ be the Eisenstein series associated with $\Phi_{w^\vee,w}$ as in \eqref{eq:def_Eisenstein_series} and define
\begin{equation*}
Z_w(na;f,w^\vee,\Phi(s))=\int_{G(F) \backslash G(\mathbb{A}_F)}f(g)E(\iota(na,g),\Phi_{w^\vee,w},s)dg.
\end{equation*}
By \autoref{thm:SiegelWeil}, our discussion so far shows that
\begin{equation*}
([\tilde{\Phi}(T,\tilde{\varphi}_{w^\vee,w},s')],\overline{\theta(f,\varphi \otimes \varphi_\infty)})=C \cdot \int_{A(\mathbb{R})^0}\int_{N(F)\backslash N(\mathbb{A})} \widetilde{\mathcal{M}}_T(na,s') Z_w(na;f,w^\vee,\Phi(s_0)) dn da,
\end{equation*}
where $s_0=(n-3)/2$ and $C=c \cdot \kappa_{n+2}^{-1}$ with $\kappa_4=2$ and $\kappa_{n+2}=1$ for $n>2$. Now \autoref{proposition:Euler_product} shows that
\begin{equation*}
Z_w(na;f,w^\vee,\Phi(s))=f_{\iota_\pi(w \otimes w_\infty)}(na) \cdot \frac{\Lambda(\pi^\vee,\chi_V,s+1/2)}{d(\chi_V,s)} \cdot \prod_v Z^0(f_v^\vee,f_v,\delta_*\Phi_{w_v^\vee,w_v},\chi_{V,v},s)
\end{equation*}
for any $s$ outside the set of poles of the Eisenstein series; the result follows from this. 

\section{An example: Products of Shimura curves} \label{section:Example_Products_Shimura_curves}

In this section we focus on a special case of the results above: when the Shimura variety $X_K$ is a product of Shimura curves attached to a quaternion algebra $B$ over $\mathbb{Q}$. This case was addressed in \cite[\textsection 4]{GarciaRegularizedLiftsI}; there we computed explicitly some currents $[\Phi(T,\varphi)]$ in terms of CM points and Hecke correspondences. The main goal is to give a concrete example where the hypotheses \ref{hyp:main_thm_pairing} hold and to prove \autoref{thm:main_thm_shimura_curves}.


Throughout this section, we fix an indefinite quaternion algebra $B$ over $\mathbb{Q}$; we assume that $B \ncong M_2(\mathbb{Q})$. We write $S$ for the set of places where $B$ ramifies and $d(B)$ for the discriminant of $B$. Denote by $n: B \rightarrow F$ the reduced norm and consider the vector space $B$ endowed with the quadratic form given by $Q(v)=n(v)$. Then $(B,Q)$ is a non-degenerate quadratic space over $\mathbb{Q}$ with signature $(2,2)$ and $\chi_V=1$. For basic results on the Shimura variety attached to $GSpin(B)$ we refer to \cite[\textsection 4.1]{GarciaRegularizedLiftsI}.

\subsection{Theta correspondence for $(GSp_4,GO(V))$} \label{subsection:theta_corr_GSp4} Let $(V,Q)$ be a quadratic vector space over a totally real field $F$. Consider the similitude groups
\begin{equation} \label{eq:def_GSp_4_GO(V)}
\begin{split}
GSp_4&=\{g \in GL_4 | {}^tg J g =\nu(g)J, \quad \nu(g) \in \mathbb{G}_m  \} \\
GO(V)&=\{g \in GL(V) | Q(gv) =\nu(g)Q(v), \quad v \in V, \quad \nu(g) \in \mathbb{G}_m  \}.
\end{split}
\end{equation}
Then $\nu:GSp_4 \rightarrow \mathbb{G}_m$ (resp. $\nu:GO(V) \rightarrow \mathbb{G}_m$) is a character with kernel equal to $Sp_4$ (resp. $O(V)$). We review some results concerning the theta correspondence for the pair $(GSp_4,GO(V))$.

\subsubsection{Representations of $GO(V)$} Let $\tilde{H}=GO(V)$ and $\tilde{H}^0=GSO(V)$ be the identity component of $\tilde{H}$; we regard $H$ and $\tilde{H}^0$ as algebraic groups over $F$. Recall that
\begin{equation}\label{eq:GSO(V)}
\tilde{H}^0 \cong \mathbb{G}_m \backslash B^\times \times B^\times.
\end{equation}
Let $\pi_v$ be a smooth irreducible representation of $\tilde{H}^0(F_v)$; for $F_v=\mathbb{R}$, this means an irreducible admissible $(\mathfrak{g},K)$-module. It follows from the isomorphism \eqref{eq:GSO(V)} that we can identify $\pi_v$ with $\pi_{1,v} \boxtimes \pi_{2,v}$, where $\pi_{i,v}$, $i=1,2$ are smooth irreducible representations of $B^\times(F_v)$ with central characters $\omega_i$ satisfying $\omega_1 \cdot \omega_2=1$. Given such $\pi_{1,v},\pi_{2,v}$, we denote by $(\pi_{1,2})_v$ the corresponding representation of $\tilde{H}^0(F_v)$. Similarly, any cuspidal automorphic representation $\pi$ of $\tilde{H}^0(\mathbb{A})$ is of the form $\pi_{1,2} \cong \pi_1 \boxtimes \pi_2$, where $\pi_1$ and $\pi_2$ are cuspidal automorphic representations of $B^\times(\mathbb{A})$ whose central characters $\omega_1, \omega_2$ satisfy $\omega_1\cdot\omega_2=1$.

Given an irreducible representation $(\pi_{1,2})_v$ of $\tilde{H}^0(F_v)$, consider the induced representation $Ind_{\tilde{H}^0(F_v)}^{\tilde{H}(F_v)}(\pi_{1,2})_v$. If $\pi_{1,v} \ncong \pi_{2,v}^\vee$, then this representation is irreducible, and we denote it by $(\pi^+_{1,2})_v$. Otherwise, $(\pi_{1,2})_v$ admits two different extensions to $\tilde{H}(F_v)$. Denote them by $(\pi^+_{1,2})_v$ and $(\pi^-_{1,2})_v$, with $(\pi^+_{1,2})_v$ the unique such extension satisfying
\begin{equation}
Hom_{H_1}((\pi^+_{1,2})_v,\mathbb{C}) \neq 0,
\end{equation}
where $H_1=O(V_1)$ and $V_1$ denotes the orthogonal complement of the vector $1 \in V(F_v)$. Denote by $\mathcal{R}_2(\tilde{H}(F_v))$ the set of irreducible admissible representations of $\tilde{H}(F_v)$ (up to isomorphism) whose restriction to $O(V(F_v))$ is multiplicity free and has an irreducible constituent that appears as a non-zero quotient of $\mathcal{S}(V(F_v)^2)$. Then we have
\begin{equation} \label{eq:local_theta_sign}
(\pi_{1,2}^+)_v \in \mathcal{R}_2(\tilde{H}(F_v)), \quad (\pi_{1,2}^-)_v \notin \mathcal{R}_2(\tilde{H}(F_v))
\end{equation}
for any irreducible representation $(\pi_{1,2})_v$ of $\tilde{H}^0(F_v)$; this is \cite[Theorem 3.4]{RobertsGlobalLpackets}.



\subsubsection{Local correspondence} Following \cite[\textsection 5]{HarrisKudlaGSp2}, we recall the extension of the local theta correspondence to similitude groups. Define
\begin{equation}
R=G(Sp_4 \times O(V))=\{(g,h) \in GSp_4 \times GO(V) \ | \ \nu(g)=\nu(h) \}.
\end{equation}
For any place $v$ of $F$, the action of $Sp_4(F_v) \times O(V(F_v))$ on $\mathcal{S}(V(F_v)^2)$ can be extended to an action of $R(F_v)$ that we still denote by $\omega$. Namely, for $h \in \tilde{H}(F_v)$, define
\begin{equation}
L(h)\varphi(x)=|\nu(h)|^{-4} \cdot \varphi(h^{-1}x)
\end{equation}
and for $(g,h) \in R(F_v)$, define
\begin{equation} \label{eq:def_extension_theta_R}
\omega(g,h)\varphi(x)=\omega\left( g \left( \begin{array}{cc} 1_2 & \\ & \nu(g)^{-1} 1_2 \end{array} \right),1 \right) L(h)\varphi(x).
\end{equation}

Assume first that $F_v=\mathbb{R}$ or $F_v$ is non-archimedean of odd residue characteristic. The results of Roberts extending the proof of the Howe duality conjecture to similitude groups show that for any irreducible representation $(\pi_{1,2})_v$ of $\tilde{H}^0(F_v)$, there exists a unique smooth irreducible representation $\theta((\pi_{1,2})_v)$ of $GSp_4(F_v)$ characterized by the condition
\begin{equation}
Hom_{R(F_v)}(\mathcal{S}(V(F_v)^2),\theta((\pi_{1,2})_v) \boxtimes (\pi_{1,2}^+)_v) \neq 0.
\end{equation}
The representation $\theta((\pi_{1,2})_v)$ is known as the local theta lift of $(\pi_{1,2})_v$. If $F_v$ is non-archimedean of residue characteristic $2$, then the same result holds provided that we assume $(\pi_{1,2})_v$ to be tempered (see \cite[Thm. 1.8]{RobertsGlobalLpackets}). 

We now describe the representation $\theta((\pi_{1,2})_v)$ for non-archimedean $v$ explicitly in some cases, following \cite{GanTakedaThetaGSp4}. To do so, we need to introduce certain induced representations of $GSp_4(F_v)$. Namely, let $T(F_v)$ be the maximal torus of diagonal matrices in $GSp_4(F_v)$. Characters of $T(F_v)$ can be identified with triples of characters $(\chi_1,\chi_2,\chi)$; the character of $T(F_v)$ corresponding to such a triple is given by
\begin{equation}
\left( \begin{array}{cccc} a_1 & & & \\ & a_2 & & \\ & & a a_1^{-1} & \\ & & & aa_2^{-1} \end{array} \right) \mapsto \chi_1(a_1)\chi(a_2)\chi(a).
\end{equation}
Let $B(F_v)=T(F_v)U(F_v)$ be a Borel subgroup of $GSp_4(F_v)$ containing $T(F_v)$ and having unipotent radical $U(F_v)$. For characters $\chi_1,\chi_2,\chi$ of $F_v^\times$, let $I_B(\chi_1,\chi_2;\chi)$ denote the normalized parabolic induction. This representation has finite length, and if the characters are unramified, then it has a unique subquotient admitting a non-zero $GSp_4(\mathcal{O}_v)$-invariant vector; we denote this subquotient by $\pi_B(\chi_1,\chi_2;\chi)$. 

Let $Q(F_v) \subset GSp_4(F_v)$ be the parabolic subgroup stabilizing a fixed line in $F_v^4$. The group $Q(F_v)$ has a Levi decomposition $Q(F_v)=L(F_v)U_Q(F_v)$, with $U_Q(F_v)$ the unipotent radical of $Q(F_v)$ and $L(F_v)$ a Levi subgroup. We have $L(F_v) \cong GL_1(F_v) \times  GL_2(F_v)$, and hence irreducible smooth representations of $L(F_v)$ are of the form $\chi \boxtimes \tau$, where $\chi$ is a character of $GL_1(F_v)$ and $\tau$ is a smooth irreducible representation of $GL_2(F_v)$. We denote by $I_{Q(F_v)}(\chi,\tau)$ the representation of $GSp_4(F_v)$ obtained by normalized parabolic induction of such a representation of $L(F_v)$.

For characters $\chi,\chi'$ of $F_v^\times$, denote by $\pi(\chi,\chi')$ the principal series representation of $GL_2(F_v)$ obtained by normalized induction from the standard Borel subgroup of upper triangular matrices.

\begin{proposition} \label{prop:local_theta_GO4} Let $F_v$ be non-archimedean and $(\pi_{1,2})_v=\pi_{1,v} \boxtimes \pi_{2,v}$ be a smooth irreducible representation of $\tilde{H}^0(F_v)$. If $F_v$ has even residue characteristic, assume that $(\pi_{1,2})_v$ is tempered.
\begin{enumerate}
\item Assume that $B_v$ is split and that $\pi_{1,v}$, $\pi_{2,v}$ are unramified and pre-unitary. Write $\pi_{1,v} \subset \pi(\chi_1,\chi_1')$ and $\pi_{2,v} \subset \pi(\chi_2,\chi_2')$ for unramified characters $\chi_i,\chi_i':F_v^\times \rightarrow \mathbb{C}^\times$ and $i=1,2$. Then
\begin{equation*}
\theta((\pi_{1,2})_v)=\pi_B(\chi_2'\chi_1',\chi_2\chi_1';\chi_1).
\end{equation*} 
\item Assume that $B_v$ is split and that $\pi_{1,v} \boxtimes \pi_{2,v}=\tau \boxtimes \tau^\vee$, where $\tau$ is a discrete series representation of $GL_2(F_v)$. Then $\theta((\pi_{1,2})_v)$ is an irreducible constituent of $I_{Q(F_v)}(1,\tau)$.
\item Assume that $B_v$ is ramified and that $\pi_{1,v} \boxtimes \pi_{2,v}=\tau \boxtimes \tau^\vee$ for some representation $\tau$ of $B_v^\times$. Let $\tau^{JL}$ be the discrete series representation of $GL_2(F_v)$ attached to $\tau$ by the local Jacquet-Langlands correspondence. Then $\theta((\pi_{1,2})_v)$ is an irreducible constituent of $I_{Q(F_v)}(1,\tau^{JL})$.
\end{enumerate}
\end{proposition}
\begin{proof}
This follows from \cite[Thm. 8.1, Thm. 8.2]{GanTakedaThetaGSp4}, together with the fact that $\theta((\pi_{1,2})_v)$ is unramified provided that $(\pi_{1,2})_v$ is unramified and pre-unitary (see \cite[Lemma 11]{HarrisSoudryTaylor}). (Note that representations of $\tilde{H}^0(F_v)$ in \cite{GanTakedaThetaGSp4} are identified with pairs $\pi_{1,v}$, $\pi_{2,v}$ of representations of $B_v^\times$ with the same central character; such a pair corresponds in our description to the representation $\pi_{1,v} \boxtimes \pi_{2,v}^\vee \cong\pi_{1,v} \boxtimes (\pi_{2,v} \otimes \omega_{\pi_{1,v}}^{-1})$.)
\end{proof}

Denote by $Si(\mathfrak{p}_v)$ the Siegel parahoric of $GSp_4(F_v)$, defined by
\begin{equation}
Si(\mathfrak{p}_v)=\left\{ \left( \begin{array}{cc} a & b \\ c & d \end{array} \right) \in GSp_4(\mathcal{O}_v) \ | \ c \in \varpi M_2(\mathcal{O}_v)  \right\}.
\end{equation}
Let $St_{GL_2}$ be the Steinberg representation of $GL_2(F_v)$, that is, the unique irreducible subrepresentation of the induced representation $\pi(|\cdot|^{1/2},|\cdot|^{-1/2})$. 
\begin{lemma} \label{lemma:Local_Siegel_fixed_vectors}
Let $(\pi_{1,2})_v=\tau \boxtimes \tau^\vee$, with $\tau$ a representation of $B_v^\times$. Assume that the Jacquet-Langlands lift of $\tau$ is of the form $\tau^{JL}=\chi St_{GL_2}$, where $\chi$ is an unramified quadratic character of $F_v^\times$. Then the space of vectors in $\theta((\pi_{1,2})_v)$ fixed by $Si(\mathfrak{p}_v)$ is one-dimensional.
\end{lemma}
\begin{proof}
With the notation of \cite[\textsection 2.2]{RobertsSchmidBook}, the representation $\theta((\pi_{1,2})_v)$ equals $\tau(S,\nu^{-1/2}\chi)$ (if $B_v$ is split) or $\tau(T,\nu^{-1/2}\chi)$ (if $B_v$ is division). Both representations have a one-dimensional space vectors fixed by $Si(\mathfrak{p}_v)$ (see \cite[Table A.15]{RobertsSchmidBook})).
\end{proof}

We now consider the description of theta lifts in terms of Langlands parameters. Let $W_{F_v}$ be the Weil group of $F_v$ and $WD_{F_v}=W_{F_v} \times SL_2(\mathbb{C})$ be the Weil-Deligne group of $F_v$. Given a smooth irreducible representation $\pi_v$ of $B^\times_v$, denote by $\phi_{\pi_v}:WD_{F_v} \rightarrow GL_2(\mathbb{C})$ the Langlands parameter of its Jacquet-Langlands lift to $GL_2(F_v)$. Note that the character $\det(\phi_{\pi_v}):W_{F_v} \rightarrow \mathbb{C}^\times$ corresponds to the central character $\omega_{\pi_v}$ of $\pi_v$ by local class field theory; in particular, $\det(\phi_{\pi_{1,v}})\det(\phi_{\pi_{2,v}})=1$ if $\omega_{\pi_{1,v}}\omega_{\pi_{2,v}}=1$. Given two representations $\pi_{1,v}$ and $\pi_{2,v}$ with $\omega_{\pi_{1,v}}\omega_{\pi_{2,v}}=1$, we obtain a Langlands parameter
\begin{equation}
\phi_{\pi_{1,v}} \oplus (\phi_{\pi_{2,v}} \otimes \det(\phi_{\pi_{1,v}})) : WD_{F_v} \rightarrow GSp_4(\mathbb{C})
\end{equation}
by composing the map $(\phi_{\pi_{1,v}},\phi_{\pi_{2,v}} \otimes \det(\phi_{\pi_{1,v}})):WD_{F_v} \rightarrow GL_2(\mathbb{C}) \times_{GL_1(\mathbb{C})} GL_2(\mathbb{C})$ with the embedding $\iota_0:GL_2(\mathbb{C}) \times_{GL_1(\mathbb{C})} GL_2(\mathbb{C}) \rightarrow GSp_4(\mathbb{C})$ given by \eqref{eq:def_iota_0}.

Given a smooth irreducible representation $\Pi$ of $GSp_4(F_v)$, denote by
\begin{equation}
\phi_\Pi:WD_{F_v} \rightarrow GSp_4(\mathbb{C})
\end{equation}
the Langlands parameter attached to it in \cite{GanTakedaLLCGSp4}.

\begin{lemma} \label{lemma:LL_param}
Let $F_v$ and $(\pi_{1,2})_v=\pi_{1,v} \boxtimes \pi_{2,v}$ be as in \autoref{prop:local_theta_GO4}. Let $\Pi=\theta((\pi_{1,2})_v)$ and let $\phi_\Pi:WD_{F_v} \rightarrow GSp_4(\mathbb{C})$ be its local Langlands parameter. Then
\begin{equation*}
\phi_\Pi=\phi_{\pi_{1,v}} \oplus (\phi_{\pi_{2,v}} \otimes \det(\phi_{\pi_{1,v}})).
\end{equation*}
\end{lemma}
\begin{proof}
This is \cite[Prop. 13.1]{GanTakedaThetaGSp4}. (The twist by $\det(\phi_{\pi_{1,v}})$ in our definition of the Langlands parameter arises from the different parametrization of representations of $\tilde{H}^0(F_v)$ used there.)
\end{proof}

Denote by $std:GSp_4(\mathbb{C}) \rightarrow GL_5(\mathbb{C})$ the five-dimensional representation of $GSp_4(\mathbb{C})$ given by the homomorphism
\begin{equation}
GSp_4(\mathbb{C}) \rightarrow PGSp_4(\mathbb{C}) \cong SO_5(\mathbb{C}) \hookrightarrow GL_5(\mathbb{C})
\end{equation}
described in \cite[\textsection A.7]{RobertsSchmidBook}. Given a smooth irreducible representation $\Pi$ of $GSp_4(F_v)$ with parameter $\phi_\Pi$, one can define a local L-factor $L(std \circ \phi_\Pi,s)$; see e.g. \cite[\textsection 4.1]{TateCorvallis}. If $\Pi'$ is an irreducible constituent of $\Pi|_{Sp_4(F_v)}$ having non-zero vectors fixed by an Iwahori subgroup of $Sp_4(F_v)$ (in particular, if $\Pi'$ is unramified), then we have
\begin{equation}
L(\Pi',\mathbbm{1},s)=L(std \circ \phi_\Pi,s),
\end{equation}
where the $L$-factor on the left hand side is defined as in \autoref{subsection:local_doubling_integrals} (see \cite[\textsection 10]{LapidRallis}).

Suppose that $\Pi=\theta((\pi_{1,2})_v)$ is one of the representations in \autoref{prop:local_theta_GO4}. Then \autoref{lemma:LL_param} implies that
\begin{equation}
L(std \circ \phi_\Pi,s)=\zeta_v(s)L(\phi_{\pi_{1,v}} \otimes \phi_{\pi_{2,v}},s).
\end{equation}
Here $\zeta_v(s)$ is the local zeta function of $F_v$, and we have
\begin{equation}
L(\phi_{\pi_{1,v}} \otimes \phi_{\pi_{2,v}},s)=L(\pi_{1,v} \times \pi_{2,v},s),
\end{equation} 
where $L(\pi_{1,v} \times \pi_{2,v},s)$ is the local $L$-function defined in \cite{JacquetPSShalika}. In particular, if $(\pi_{1,2})_v$ is either unramified and pre-unitary or one of the representations in \autoref{lemma:Local_Siegel_fixed_vectors} and $\Pi'$ is a constituent of $\theta(\pi_{1,2})_v)|_{Sp_4(F_v)}$ having non-zero Iwahori-fixed vectors, then
\begin{equation}
L(\Pi',\mathbbm{1},s)=\zeta_v(s)L(\pi_{1,v} \times \pi_{2,v},s).
\end{equation}

\subsubsection{Global correspondence} The restricted tensor product over all places $v$ of $F$ of the representations $\omega_v$ of $R(F_v)$ in \eqref{eq:def_extension_theta_R} gives a representation of $R(\mathbb{A}_F)$ on $\mathcal{S}(V(\mathbb{A}_F)^2)$. For $(g,h) \in R(\mathbb{A}_F)$ and $\varphi \in \mathcal{S}(V(\mathbb{A}_F)^2)$, the theta function
\begin{equation}
\theta(g,h;\varphi)=\sum_{(v,w) \in V(F)^2} \omega(g,h)\varphi(v,w)
\end{equation}
is left invariant under $R(F)$ (see \cite[Lemma 5.1.7] {HarrisKudlaGSp2}). Let $f$ be a cusp form on $\tilde{H}(\mathbb{A}_F)$. Given $g \in GSp_4(\mathbb{A}_F)$, choose $h \in \tilde{H}(\mathbb{A}_F)$ with $\nu(g)=\nu(h)$ and define
\begin{equation}
\theta(g;f,\varphi)=\int_{O(V)(F) \backslash O(V)(\mathbb{A}_F)} f(h'h)\theta(g,h'h;\varphi)dh'.
\end{equation}
Clearly the integral does not depend on the choice of $h$. It defines an automorphic form on $GSp_4(\mathbb{A}_F)$ and, if $f$ has a central character $\chi$, then $\theta(f,\varphi)$ also has central character $\chi$ (\cite[Lemma 5.1.9]{HarrisKudlaGSp2}). Given a cuspidal automorphic representation $\pi \subset \mathcal{A}_0(\tilde{H})$ of $\tilde{H}(\mathbb{A})$, the space
\begin{equation}
\Theta(\pi):=\{ \theta(f,\varphi) \ | \ f \in \pi, \varphi \in \mathcal{S}(V(\mathbb{A}_F)^2) \} \subset \mathcal{A}(GSp_{4,F})
\end{equation}
is stable under the action of $GSp_4(\mathbb{A}_F)$ by right translation and is known as the global theta lift of of $\pi$.

Let $\pi_{1,2} \subset \mathcal{A}_0(\tilde{H}^0)$ be a tempered cuspidal automorphic representation of $\tilde{H}^0(\mathbb{A})$. Thus $\pi_{1,2}=\pi_1 \boxtimes \pi_2$, where $\pi_1$ and $\pi_2$ are tempered cuspidal automorphic representations of $B(\mathbb{A})^\times$ whose central characters $\omega_{\pi_1}$ and $\omega_{\pi_2}$ satisfy $\omega_{\pi_1}\omega_{\pi_2}=1$. Write $\pi_{1,2}=\otimes'_v (\pi_{1,2})_v$ and consider the admissible irreducible representation
\begin{equation}
\pi_{1,2}^+:=\otimes'_v (\pi_{1,2}^+)_v
\end{equation}
of $\tilde{H}(\mathbb{A}_F)$. By \cite[Thm. 7.1]{RobertsGlobalLpackets}, $\pi_{1,2}^+$ is a cuspidal automorphic representation that appears with multiplicity one in $\mathcal{A}_0(\tilde{H})$; we denote by $V_{\pi_{1,2}^+} \subset \mathcal{A}_0(\tilde{H})$ the corresponding space of automorphic forms and by  $V^0_{\pi_{1,2}^+} \subset \mathcal{A}_0(\tilde{H}^0)$ the space of automorphic forms obtained by restricting forms in $V_{\pi_{1,2}^+}$ to $\tilde{H}^0(\mathbb{A}_F)$. Define
\begin{equation}
\Theta(\pi_1 \boxtimes \pi_2)=\Theta(\pi_{1,2}^+).
\end{equation}
\begin{theorem} \cite[Thm 8.3, Thm. 8.5]{RobertsGlobalLpackets} \label{thm:Roberts}
If $\pi_1 \ncong \pi_2^\vee$, then $\Theta(\pi_1 \boxtimes \pi_2) \neq 0$ is a cuspidal automorphic representation of $GSp_4(\mathbb{A}_F)$. We have
\begin{equation*}
\Theta(\pi_1 \boxtimes \pi_2)=\otimes'_v \theta((\pi_{1,2})_v).
\end{equation*}
Moreover, the representation $\Theta(\pi_1 \boxtimes \pi_2)$ appears with multiplicity one in the space of cusp forms with central character $\omega_{\pi_1}$.
\end{theorem}

For the rest of this paper we set $(V,Q)=(B,n)$. Fix once and for all an isomorphism $\iota : B \otimes_{\mathbb{Q}} \mathbb{A}^S \cong M_2(\mathbb{A}^S)$. For $p$ a prime dividing $d(B)$, denote by $\mathcal{O}_{B,p}$ the maximal order of $B \otimes_{\mathbb{Q}} \mathbb{Q}_p$. Define
\begin{equation}
\hat{\mathcal{O}}_B= \iota^{-1}(\prod_{p \notin S} M_2(\mathbb{Z}_p)) \times \prod_{p \in S} \mathcal{O}_{B,p}, \qquad K_B=\hat{\mathcal{O}}_B^\times.
\end{equation}
Then $\hat{\mathcal{O}}_B$ is a maximal order of $B \otimes_{\mathbb{Q}} \mathbb{A}_f$ and $K_B$ is a maximal compact subgroup of $B(\mathbb{A}_f)^\times$.

From now on, fix tempered cuspidal automorphic representations $\pi_1 \cong \otimes'_v \pi_{1,v}$ and $\pi_2 \cong \otimes'_v \pi_{2,v}$ of $B(\mathbb{A})^\times$ with trivial central character such that $\pi_1 \ncong \pi_2^\vee$ and automorphic forms $f_i \in \pi_i$ for $i=1,2$. Assume that $f_1=\otimes_v f_{1,v}$ and $f_2=\otimes_v f_{2,v}$ are pure tensors and that:
\begin{itemize}
\item $\pi_{1,\infty}=\pi_{2,\infty}=\mathcal{D}_2$ is the discrete series of $GL_2(\mathbb{R})$ of weight $2$ with trivial central character, $f_{1,\infty} \in \pi_{1,\infty}$ has weight $2$ and $f_{2,\infty} \in \pi_{2,\infty}$ has weight $-2$;

\item $\pi_{1,v}=\pi_{2,v}=\mathbbm{1}$ is the trivial representation of $(B_v)^\times$ if $B$ is ramified at $v$;

\item $(\pi_i)_f^{K_B}= \mathbb{C} \cdot \otimes_{v \nmid \infty} f_{i,v}$ for $i=1,2$.
\end{itemize}

Define
\begin{equation}
\Pi=\Theta(\pi_1 \boxtimes \pi_2) \subset \mathcal{A}_0(GSp_{4,F})
\end{equation}
to be the global theta lift to $GSp_4(\mathbb{A})$; note that $\Pi$ has trivial central character. Let $K^{(2)}=\prod_{p}K^{(2)}_p$ be the open compact subgroup of $GSp_4(\mathbb{A}_f)$ defined by
\begin{equation}
K^{(2)}_p= \left\{ \begin{array}{ccc} GSp_4(\mathbb{Z}_p), & \quad \text{if} \quad p \nmid disc(B),  \\
Si(p), & \quad \text{if} \quad p \ | \ disc(B).
\end{array}  \right.
\end{equation}
It follows from \autoref{lemma:Local_Siegel_fixed_vectors} that the space of vectors in $\Pi_f$ fixed by $K^{(2)}$ is one dimensional; for every finite place $v$, we fix a vector $f_v \in \Pi_v$ such that
\begin{equation}
\Pi_f^{K^{(2)}}=\mathbb{C} \cdot \otimes_{v \nmid \infty} f_v.
\end{equation}

At the archimedean place, we have $\Pi_\infty \cong \pi_{\infty,1}$, where $\pi_{\infty,1}$ is the representation of $PGSp_4(\mathbb{R})$ defined in \autoref{subsection:reps_and_cohomology}. Lemma \ref{lemma:archimedean_K_type_one} shows that the $K_\infty$ type $\sigma_{2,0}$ occurs with multiplicity one in $\Pi_\infty$, and we choose a highest weight vector $f_\infty \in \Pi_\infty$. We fix an automorphic form $f \in \Pi$ such that under some isomorphism $\Pi \cong \otimes_v' \Pi_v$ we have
\begin{equation} \label{eq:def_f}
f = \otimes_v f_v
\end{equation}
with $f_v \in \Pi_v$ as above (note that such an automorphic form $f$ is unique up to scalar multiplication by the multiplicity one part of \autoref{thm:Roberts}).

\subsection{Local data} \label{subsection:local_data} The results of Roberts quoted above imply that $\Theta(\Pi) \neq 0$. However, to apply \autoref{thm:Rallis_pairing} we need to show that $\theta(f,\varphi) \neq 0$ for our chosen $f$ and a specific Schwartz function $\varphi \in \mathcal{S}(V(\mathbb{A})^2)$. To achieve this, one can try using the Rallis inner product formula for the Petersson norm of $\theta(f,\varphi)$, which reduces the non-vanishing of $\theta(f,\varphi)$ to the non-vanishing of local doubling integrals such as those studied in \autoref{subsection:Local_zeta_integrals}. Since the explicit evaluation of such integrals (especially in the archimedean case) can be quite complicated, we have chosen a different approach, which proceeds by using results of \cite{WatsonThesis} and an argument using seesaw duality. As a byproduct, once we have shown that $\theta(f,\varphi) \neq 0$ for some $\varphi$, we can use the Rallis inner product formula to give a global proof that certain local zeta integrals do not vanish.

We start by reviewing some results of \cite{WatsonThesis} on the theta correspondence for the dual pair $(GL_2,GO(V))$. For $\varphi \in \mathcal{S}(V(\mathbb{A}))$, consider the theta lift
\begin{equation}
\theta(\cdot,\varphi): \mathcal{A}_0(GL_2) \rightarrow \mathcal{A}(\tilde{H})
\end{equation}
defined by
\begin{equation}
\theta(f,\varphi)(h)=\int_{SL_2(\mathbb{Q})\backslash SL_2(\mathbb{A})} \overline{f(gg')} \theta(gg',h;\varphi)dg,
\end{equation}
where $h \in \tilde{H}(\mathbb{A})$ and $g \in GL_2(\mathbb{A})$ is such that $\det(g)=\nu(h)$.

Let $\pi=\pi_f \otimes \pi_\infty$ be an automorphic representation of $B(\mathbb{A})^\times$ with trivial central character. Set
\begin{equation}
k_\pi=\left\{ \begin{array}{cc} k,& \quad \text{if } \pi_\infty \text{ is a discrete series of weight } k \geq 2; \\
0,& \quad \text{otherwise.}
 \end{array} \right.
\end{equation}
Assume that 
\begin{equation}
\dim_\mathbb{C} (\pi_f)^{K_B}=1.
\end{equation}
Denote by $\pi^{JL}$ the automorphic representation of $GL_2(\mathbb{A})$ attached to $\pi$ by the Jacquet-Langlands correspondence. Given such a $\pi$, there exist automorphic forms $f_0 \in \pi$ and $f_0^{JL} \in \pi^{JL}$, unique up to multiplication by scalars, such that
\begin{itemize}
\item $f_0$ (resp. $f_0^{JL}$) is right invariant under $K_B$ (resp. $K_0(disc(B))$);
\item Both $f_0$ and $f_0^{JL}$ have weight $k_\pi$ under the right action by $SO(2)\subset GL_2(\mathbb{R})$.
\end{itemize}

Let $v$ be a place of $\mathbb{Q}$.
\begin{itemize}
\item Suppose that $v=p$ is finite. Set $\mathcal{O}_{B,v}=\hat{\mathcal{O}}_{B} \otimes_\mathbb{Z} \mathbb{Z}_p$ and let $\varphi^0_v=\mathbbm{1}_{\mathcal{O}_{B,v}} \in \mathcal{S}(B_v)$ be the characteristic function of $\mathcal{O}_{B,v} \subset B_v$.

\item Suppose that $v=\infty$. Let $\varphi^{(0)}_v(x)=\varphi^0(x,z_0)$, where $\varphi^0(x,z)$ is the Siegel gaussian. For $k$ a positive integer, let $\varphi_v^{(k)}(x)$ be the Schwartz function corresponding to
\begin{equation}
\varphi^{(k)}_\infty \left( \left( \begin{array}{cc} a & b \\ c & d \end{array} \right) \right)=(a-ib+ic+d)^k e^{-\pi (a^2+b^2+c^2+d^2)} \in \mathcal{S}(M_2(\mathbb{R}))
\end{equation}
under the identification $V_\mathbb{R}=B_\mathbb{R} \cong M_2(\mathbb{R})$ chosen above.
\end{itemize}

For $k$ a non-negative integer, define a global Schwartz function $\varphi^{(k)} \in \mathcal{S}(V(\mathbb{A}))$ by
\begin{equation} \label{eq:def_varphi_(k)}
\varphi^{(k)} = \otimes_{v \neq \infty} \varphi_v^0 \otimes \varphi_\infty^{(k)}.
\end{equation}
Recall that there is a surjection $\iota:B(\mathbb{A})^\times \times B(\mathbb{A})^\times \rightarrow \tilde{H}^0(\mathbb{A})$ . Given an automorphic form $f \in \mathcal{A}(\tilde{H}^0)$, we write $f(h_1,h_2)$ for the value of $f \circ \iota$ on $(h_1,h_2) \in (B(\mathbb{A})^\times)^2$.

\begin{proposition} \cite[Thm. 1]{WatsonThesis} \label{prop:Watson}
Let $\pi$ be an automorphic representation of $B(\mathbb{A})^\times$ such that $\dim_\mathbb{C} (\pi_f)^{K_B}=1$. Then
\begin{equation}
\theta(f_0^{JL},\varphi^{(k_\pi)})(h_1,h_2)=\alpha \cdot f_0(h_1)\overline{f_0(h_2)}
\end{equation}
for some $\alpha \in \mathbb{C}^\times$.
\end{proposition}

Using this result, we can now choose a Schwartz form $\varphi \in \mathcal{S}(V(\mathbb{A})^2)$ such that $\theta(f,\varphi) \neq 0$. Namely, set $\varphi=\otimes_v \varphi_v$, with
\begin{equation} \label{eq:def_varphi_example}
\varphi_v(x,y)= \left\{ \begin{array}{cc} \varphi^0_v(x)\varphi^0_v(y),& \quad \text{if} \quad v \neq \infty; \\
\varphi_\infty^{(0)}(x) \varphi_\infty^{(2)}(y),& \quad \text{if} \quad v = \infty. \end{array} \right.
\end{equation}

Using the Iwahori decomposition of $Si(p)$ one shows that the Schwartz function $\otimes_{v \nmid \infty} \varphi_v$ is invariant under the group $K^{(2)}$.

\begin{proposition} \label{proposition:non_vanishing}
With the notations above, $\theta(f,\varphi) \neq 0$.
\end{proposition}
\begin{proof}
Write $f_{1,2}=f_1 \boxtimes f_2$; we regard $f_{1,2}$ as a cusp form on $\tilde{H}^0(\mathbb{A})$. For $g \in GSp_4(\mathbb{A})$, choose $h \in \tilde{H}^0(\mathbb{A})$ such that $(g,h) \in R(\mathbb{A})$ and consider the function
\begin{equation*}
F(g)=\int_{SO(V)(\mathbb{Q}) \backslash SO(V)(\mathbb{A})}\overline{f_{1,2}(hh')}\theta(g,hh';\varphi)dh.
\end{equation*}
Then $F$ is an automorphic form on $GSp_4(\mathbb{A})$. Moreover, by \eqref{eq:local_theta_sign}, the function $F$ belongs to $\Pi=\Theta(\pi_1 \boxtimes \pi_2)$. By our choice of $\varphi$, the function $F$ is right invariant under $K^{(2)}$, and defines a highest weight vector in the $K$-type $\sigma_{(2,0)} \subset \Pi_\infty$. It follows that
\begin{equation*}
F=\alpha \cdot f
\end{equation*}
for some $\alpha \in \mathbb{C}$; in particular, to prove the claim it suffices to show that $F$ is not identically zero.

Recall the embedding $\iota_0:SL_2 \times SL_2 \rightarrow Sp_4$ defined in \eqref{eq:def_iota_0}. For a pair of cusp forms $s_1,s_2 \in \mathcal{A}_0(GL_2)$, consider the integral
\begin{equation*}
I(F,s_1,s_2)=\int_{(SL_2(\mathbb{Q}) \backslash SL_2(\mathbb{A}))^2}F(\iota_0(g_1,g_2))\overline{s_1(g_1)} \  \overline{s_2(g_2)}dg_1dg_2.
\end{equation*}
We need to show that $I(F,s_1,s_2) \neq 0$ for some pair $s_1,s_2$. Note that
\begin{equation*}
\theta(\iota_0(g_1,g_2),h;\varphi)=\theta(g_1,h;\varphi^{(0)}) \cdot \theta(g_2,h;\varphi^{(2)}),
\end{equation*}
with $\varphi^{(k)} \in \mathcal{S}(V(\mathbb{A}))$ given by \eqref{eq:def_varphi_(k)}. Interchanging the integrals we obtain
\begin{equation*}
I(F,s_1,s_2)=\int_{SO(V)(\mathbb{Q}) \backslash SO(V)(\mathbb{A})} \overline{f_{1,2}(h)} \theta(s_1,\varphi^{(0)})(h) \cdot \theta(s_2,\varphi^{(2)})(h)dh.
\end{equation*}
Choose $s_1$ and $s_2$ to be cuspidal eigenforms on $GL_2(\mathbb{A})$ of level $d(B)$; assume that $s_1$ has weight $0$ and that $s_2$ corresponds to a holomorphic cusp form of weight $2$. By \autoref{prop:Watson}, we have
\begin{equation*}
\begin{split}
\theta(s_1,\varphi^{(0)})(h_1,h_2)&=s_1'(h_1) \overline{s_1'(h_2)} \\
\theta(s_2,\varphi^{(2)})(h_1,h_2)&=s_2'(h_1) \overline{s_2'(h_2)}.
\end{split}
\end{equation*}
for certain (non-zero) automorphic forms $s_1'$ and $s_2'$ on $B(\mathbb{A})^\times$ that are right invariant under $K_B$. We conclude that
\begin{equation*}
I(F,s_1,s_2)=C \cdot \int_{X_{B,K}}s_1'(h)\overline{f_1(h)}s_2'(h)dh \int_{X_{B,K}}\overline{s_1'(h)} \ \overline{f_2(h)} \ \overline{s_2'(h)}dh
\end{equation*}
for some non-zero constant $C$. Note that the functions $\overline{f_1(h)}s_2'(h)$ and $\overline{f_1(h)} \ \overline{s_2'(h)}$ on $X_K$ are not identically zero. Moreover, by properties of the Jacquet-Langlands correspondence, the $L^2$-closure of the space generated by the functions $s_1'$ (for varying $s_1$) equals
\begin{equation*}
\{f \in L^2(X_{B,K}) \ | \ \int_{X_{B,K}}f(h)dh=0 \} \subset L^2(X_{B,K}).
\end{equation*}
It follows easily from this that for any fixed $f_1,f_2,s_2$ one can find $s_1$ such that $I(F,s_1,s_2) \neq 0$. Namely, note that the Riemann surface $X_{B,K}$ has genus at least $2$ (since $f_1$ and $\overline{f_2}$ correspond to independent holomorphic differentials on $X_{B,K}$). Hence there exists $z_0 \in X_{B,K}$ such that $s_2(z_0)=0$. Let $z_1 \in X_{B,K}$ such that $f_1(z_1)f_2(z_1)s_2'(z_1) \neq 0$. The claim follows by picking a sequence of functions $s_1$ approaching $\delta_{z_1}-\delta_{z_0}$.
\end{proof}

We note the following corollary of the previous proof.

\begin{corollary}
The automorphic form
\begin{equation*}
f \circ \iota_0 \in \mathcal{A}(SL_2 \times SL_2)
\end{equation*}
is not identically zero. More precisely, for any holomorphic cuspidal eigenform $s_2$ on $GL_2(\mathbb{A})$ of weight $2$ and level $d(B)$, we can find a cusp form $s_1$ on $GL_2(\mathbb{A})$ of weight $0$ and level $d(B)$ such that
\begin{equation*}
\int_{(SL_2(\mathbb{Q}) \backslash SL_2(\mathbb{A}))^2}f(\iota_0(g_1,g_2)) \overline{s_1(g_1)} \ \overline{s_2(g_2)}dg_1dg_2 \neq 0.
\end{equation*}
\end{corollary}

Finally, for $\varphi \in \mathcal{S}(V(\mathbb{A}_f)^2)$, recall the theta lift of $f$ valued in differential forms
\begin{equation*}
\theta(f,\varphi \otimes \varphi_\infty)_K \in \mathcal{A}^{1,1}(X_K)
\end{equation*}
defined by \eqref{eq:def_theta_lift_n-1_K}. Our last goal in this section is to show that this lift does not vanish. To do so, we need to relate the Kudla-Millson form $\varphi_{KM}$ to our archimedean Schwartz form $\varphi^{(2)}$; this is the purpose of the following lemma.

\begin{lemma}
Let 
\begin{equation*}
\varphi_{KM} \in [\mathcal{S}(V(\mathbb{R})) \otimes \wedge^2 \mathfrak{p}^*]^{K_{z_0}}
\end{equation*}
be the Kudla-Millson Schwartz form. Then there exists a linear functional $l:\wedge^2 \mathfrak{p}^* \rightarrow \mathbb{C}$ such that
\begin{equation*}
(1 \otimes l)(\varphi_{KM})=\varphi^{(2)} \in \mathcal{S}(V(\mathbb{R})).
\end{equation*}
\end{lemma}
\begin{proof}
We recall the definition of $\varphi_{KM}$. Consider the orthogonal basis of $V(\mathbb{R})\cong M_2(\mathbb{R})$ given by
\begin{equation*}
\begin{split}
v_1= \left( \begin{array}{cc} 1 & \\ & 1 \end{array} \right) \qquad & v_2= \left( \begin{array}{cc}  & -1 \\ 1 &  \end{array} \right) \\
v_3= \left( \begin{array}{cc}  & 1 \\ 1 &  \end{array} \right) \qquad & v_4= \left( \begin{array}{cc} 1 & \\ & -1 \end{array} \right).
\end{split}
\end{equation*}
Note that $Q(v_1,v_1)=Q(v_2,v_2)=1$ and $Q(v_3,v_3)=Q(v_4,v_4)=-1$ and let $z_0 \in \mathbb{D}$ be the plane spanned by $\{v_3,v_4\}$. This basis gives an isomorphism $SO(V(\mathbb{R})) \cong SO(2,2)$ carrying $K_{z_0}$ to $SO(2) \times SO(2)$. Denote the coordinates of $v \in V(\mathbb{R})$ with respect to this basis by $x_i(v)$, $1\leq i \leq 4$. Given $v=\left( \begin{smallmatrix} a&b\\ c&d \end{smallmatrix}\right)$, we have
\begin{equation*}
a-ib+ic+d=x_1(v)+ix_2(v)
\end{equation*}
and hence the value of $\varphi^{(2)}$ at $z_0$ is given by
\begin{equation*}
\varphi^{(2)}=(x_1^2-x_2^2+2ix_1x_2) \cdot \varphi_0 \in \mathcal{S}(V(\mathbb{R})),
\end{equation*}
where we write $\varphi_0(v)=e^{-\pi (v,v)_{z_0}}$.

To describe the form $\varphi_{KM}$, note that the basis above identifies $\mathfrak{p}_0 \cong Hom_\mathbb{R}(z_0,z_0^\perp)$. We write $X_{ij}$, $1\leq i \leq 2$, $3 \leq j \leq 4$ for the resulting basis of $\mathfrak{p}$ (defined by $X_{ij}(v_k)=\delta_{jk}v_i$) and $\omega_{ij} \in \mathfrak{p}^*$ for its dual basis. Then (see \cite[\textsection 5]{KudlaMillson3})
\begin{equation*}
\begin{split}
\varphi_{KM}(\cdot,z_0)&=\frac{1}{4}\prod_{j=3}^4 \sum_{i=1}^2 \left(x_i -\frac{1}{2\pi} \frac{\partial}{\partial x_i} \right) \varphi_0 \otimes \omega_{ij} \\
&= \left( x_1^2-\frac{1}{4\pi} \right)\varphi_0 \otimes (\omega_{13} \wedge  \omega_{14}) + \left( x_2^2-\frac{1}{4\pi} \right)\varphi_0 \otimes (\omega_{23} \wedge  \omega_{24})\\
&\quad + x_1x_2 \varphi_0 \otimes (\omega_{13} \wedge \omega_{24}+\omega_{23} \wedge \omega_{14}).
\end{split}
\end{equation*}
Thus the claim follows if we set $l(\omega_{13} \wedge  \omega_{14})=1$, $l(\omega_{23} \wedge  \omega_{24})=-1$ and $l(\omega_{13} \wedge \omega_{24})=l(\omega_{23} \wedge \omega_{14})=i$.
\end{proof}

Together with \autoref{proposition:non_vanishing}, this implies the following corollary.

\begin{corollary}
Let $\varphi=\otimes_{v \nmid \infty } \varphi_v \in \mathcal{S}(V(\mathbb{A}^\infty)^2)$ with $\varphi_v$ as in \eqref{eq:def_varphi_example}. Then the theta lift $\theta(f,\varphi \otimes \varphi_\infty)$ defines a non-zero form in $\mathcal{A}^{1,1}(X_K)$.
\end{corollary}
In particular, by the Rallis inner product formula, the zeta integrals $Z^0_v(f^\vee_v,f_v,\delta_*\Phi,1,s)$ in Theorem 1.1 are non-vanishing at $s=-1/2$. It remains to check that hypotheses \ref{hyp:main_thm_pairing} hold for the representation $\Pi$ and our choice of $f$ and $\tilde{\varphi}=\varphi$; the results above show that this is the case for hypotheses (1), (2), (4) and (5). Setting $K^\infty=K^{(2)}$ and $\sigma^\infty$ to be the trivial representation of $K^\infty$, one sees that (3) holds by \autoref{lemma:Local_Siegel_fixed_vectors}. Now \autoref{thm:main_thm_shimura_curves} follows from Theorem 1.1.


%


\bibliographystyle{abbrv}
\bibliography{refs} 

\vspace{1cm}

\author{\noindent Department of Mathematics,
  South Kensington Campus,
  Imperial College London \\
  London, UK, SW7 2AZ \\
  \texttt{l.garcia@imperial.ac.uk}}

\end{document}